\documentclass[12pt]{amsart}%
\usepackage{subfig}
\usepackage{graphicx}
\usepackage[small,nohug,heads=littlevee]{diagrams}
\usepackage{tikz}
\usetikzlibrary{matrix,arrows,decorations.pathmorphing}
\usepackage{amscd}
\usepackage{geometry}
\usepackage{amsmath}
\usepackage{amsfonts}
\usepackage{amssymb}%
\usepackage{enumerate}
\usepackage{extpfeil}
\usepackage[hidelinks]{hyperref}
\usepackage[section]{placeins}
\providecommand{\U}[1]{\protect\rule{.1in}{.1in}}
\newtheorem{theorem}{Theorem}[section]
\theoremstyle{plain}

\newtheorem{claim}{Claim}

\newtheorem{corollary}[theorem]{Corollary}

\newtheorem{lemma}[theorem]{Lemma}
\newtheorem{proposition}[theorem]{Proposition}
\theoremstyle{definition}
\newtheorem{definition}{Definition}[section]

\newtheorem{question}{Question}
\newtheorem{remark}{Remark}

\numberwithin{equation}{section}

\usepackage [autostyle, english = american]{csquotes}
\MakeOuterQuote{"}

\makeatletter
\def\oversortoftilde#1{\mathop{\vbox{\m@th\ialign{##\crcr\noalign{\kern3\p@}%
      \sortoftildefill\crcr\noalign{\kern3\p@\nointerlineskip}%
      $\hfil\displaystyle{#1}\hfil$\crcr}}}\limits}

\def\sortoftildefill{$\m@th \setbox\z@\hbox{$\braceld$}%
  \braceld\leaders\vrule \@height\ht\z@ \@depth\z@\hfill\braceru$}

\makeatother

\begin{document}
\title[Contractible open manifolds which embed in no compact, LC, 1-LC space]{Contractible open manifolds which embed in no compact, locally connected and locally
1-connected metric space}
\author{Shijie Gu}
\address{Department of Mathematical Sciences\\
Central Connecticut State University, New Britain, CT 06053}
\email{sgu@ccsu.edu}
\thanks{}
\date{May 2nd, 2020}

\begin{abstract}
This paper pays a visit to a famous contractible open 3-manifold $W^3$ proposed by R. H. Bing in 1950's. By the finiteness theorem \cite{Hak68}, Haken proved that $W^3$ can embed in no compact 3-manifold. However, until now, the question about whether $W^3$ can embed in a more general compact space such as a compact, locally connected and locally 1-connected metric 3-space was unknown. Using the techniques developed in Sternfeld's 1977 PhD thesis \cite{Ste77}, we answer the above question in negative. Furthermore, it is shown that $W^3$ can be utilized to produce counterexamples for every contractible open $n$-manifold ($n\geq 4$) embeds in a compact, locally connected and locally 1-connected metric $n$-space.
\end{abstract}
\maketitle

\section{Introduction}
Counterexamples for every open 3-manifold embeds in a compact 3-manifold have been discovered for over 60 years. Indeed, there are plenty of such examples even for open manifolds which are algebraically very simple (e.g., contractible). A rudimentary version of such examples can be traced back to \cite{Whi35} (the first stage of the construction is depicted in Figure \ref{whitehead}) where Whitehead surprisingly found the first example of a contractible open 3-manifold different from $\mathbb{R}^3$. However, the Whitehead manifold does embed in $S^3$. In 1962, Kister and McMillan noticed the first counterexample in \cite{KM62} where they proved that an example proposed by Bing (see Figure \ref{3_1knot}) doesn't embed in $S^3$ although every compact subset of it does. In the meantime, they conjectured that Bing's example is a desired counterexample, i.e., such example embeds in no compact 3-manifold. This conjecture was confirmed later by Haken using his famous finiteness theorem \cite{Hak68} stating that there is an upper bound on the number of incompressible nonparallel surfaces in a compact 3-manifold. Similar examples can readily derive from Haken's finiteness theorem (or see \cite[Thm. 2.3]{MW79}). In 1977, an interesting example (see Figure \ref{sternfeld}) was given in Sternfeld's PhD dissertation \cite{Ste77}. Instead of using Haken's finiteness theorem, Sternfeld applied covering space theory to produce a contractible open $n$-manifold ($n\geq 3$) that embeds in no compact $n$-manifold\footnote{It doesn't appear that Haken's finiteness theorem can be used to produce high-dimensional examples.}.  His constructions can be viewed as a modification of Bing's\footnote{A connection between Bing's and Sternfeld's examples are illustrated in \S \ref{section: questions}.}, but he claimed that his examples cannot embed as an open subset in any compact, locally connected and locally 1-connected metric space, which is much more general than a compact manifold. More importantly, at the time of writing, Sternfeld's constructions are the only known examples of such phenomenon in high dimensions.

\begin{remark}
There is an error in Sternfeld's dissertation which directly affects his whole argument. In the process of proving our main theorem, we correct this error, thereby, confirming the validity of his example (see Remark \ref{Error} in \S \ref{section: The surjection} for details).
\end{remark}

It is natural to ask if Bing's example can embed in a more general compact space, say, a compact absolute neighborhood retract or compact, locally connected
and locally 1-connected 3-dimensional metric space. Here we answer the above question in negative.
\begin{theorem}\label{Thm: W^3 embeds in no compact ANR}
$W^3$ embeds as an open subset in no compact, locally connected, locally 1-connected metric space. In particular, $W^3$ embeds in no compact $3$-manifold. 
\end{theorem}

Making use of the high-dimensional construction developed in \cite{Ste77}, we extend Theorem \ref{Thm: W^3 embeds in no compact ANR} to all finite dimensions.
\begin{theorem}\label{Thm: high dimensional collection}
There exists a contractible open $n$-manifold $W^n$ ($n\geq 4$) which embeds as an open subset in no compact, locally connected, locally 1-connected metric $n$-space. Hence, $W^n$ embeds in no compact $n$-manifold.
\end{theorem}

The strategy of our proof heavily relies on the techniques and results from Sternfeld's dissertation \cite{Ste77}. Succinctly speaking, the key is to show that the union of $W^3$ and a 3-ball (advertised as a knot complement $K_j$) has a finite cover which contains infinitely many pairwise disjoint incompressible surfaces. Many results from \cite{Ste77} will not be re-proved here, but we will take shortcuts afforded by knot theory and software GAP \cite{GAP18} in this work.

The outline of this paper is: \S \ref{section: The constructiion of a 3-dimensional example} gives a detailed review of the construction of Bing's example and discusses its cruical connection with a knot space $K_j$. That is, showing Bing's example can embed in no compact, locally connected and 
locally 1-connected metric space is equivalent to showing $\pi_1(K_j)$ is not finitely generated. Towards that goal, in \S \ref{section: A presentation} we find the Wirtinger presentation of $\pi_1(K_j)$ and in \S \ref{section: The surjection}, we define an important surjection of $\pi_1(K_j)$ onto $\mathbb{A}_5$. Meanwhile, we fix an error in Sternfeld's dissertation. \S \ref{section: properties of cube hole} paves the road for \S \ref{section: proof of proposition} by showing that the key ingredient is to focus on an object called a cube with a trefoil-knotted hole.  \S \ref{section: proof of proposition} proves Theorem \ref{Thm: W^3 embeds in no compact ANR} by using results obtained from \S \ref{section: The constructiion of a 3-dimensional example}-\S \ref{section: properties of cube hole}. The proof of Theorem \ref{Thm: high dimensional collection} is presented at the end of this section. In \S \ref{section: questions}, we discuss some related questions of this work.

\section{The construction of a 3-dimensional example}\label{section: The constructiion of a 3-dimensional example}

First, we reproduce the example originially proposed by Bing, i.e., a 3-dimensonal contractible open manifold $W^3$. Let $\{T_l|l = 0,1,2,\dots \}$ be a collection of disjoint solid tori standardly embedded in $S^3$. Let the solid torus $T_l'$ be embedded in $\operatorname{Int}T_l$ as in Figure \ref{3_1knot}.\footnote{Changing the cube with a trefoil-knotted hole $C_l$ as shown in Figure \ref{3_1knot} can result in different contractible open manifold. For instance, one can replace $C_l$ by a cube with a square-knotted hole. Proposition \ref{Prop: W is contractible} is true for all contractible manifolds constructed in such fashion.} Let the oriented simple closed curve $\alpha_l$, $\beta_l$, $\gamma_l$ and $\delta_l$ be as shown in Figure \ref{3_1knot}. The curves $\alpha_l$ and $\beta_l$ are transverse in $\partial T_l$, and meet at the point $q_l \in \partial T_l$. In a similar fashion, the curves $\gamma_l$ and $\delta_l$ are transverse in $\partial T_l'$, and meet at the point $p_l \in \partial T_l'$. For $l \geq 1$, let $L_l = T_l \backslash \operatorname{Int}T_l'$. Define an embedding $h_{l+1}^{l}: T_l \to T_{l+1}$ so that $T_l$ is carried onto $T_{l+1}'$ with $h_{l+1}^{l}(\alpha_l) = \delta_{l+1}$ and $h_{l+1}^{l}(\beta_l) = \gamma_{l+1}$. $W^3$ is the direct limit of the $T_l$'s and denoted as $W^3 = \lim\limits_{l\to \infty}(T_l,h_{l+1}^{l})$. That is equivalent to view $W^3$ as the quotient space: $\sqcup_l T_l \xrightarrow{q} W^3$, where $\sqcup_l T_l$ is the disjoint union of the $T_l$'s and $q$ is the quotient map induced by the relation $\sim$ on $\sqcup_l T_l$. If $x\in T_i$ and $y\in T_j$, then $x \sim y$ iff there exists a $k$ larger than $i$ and $j$ such that $h_{k}^{i}(x) = h_{k}^{j}(y)$, where $h_{t}^{s} = h_{t}^{t-1} \circ h_{t-1}^{t-2} \circ \cdots \circ h_{s+2}^{s+1} \circ h_{s+1}^{s}$ for $t > s$. Let $\iota_l: T_l \hookrightarrow \sqcup_l T_l$ be the obvious inclusion map. The composition $q \circ \iota_l$ embeds $T_l$ in $W^3$ as a closed subset. The injectivity follows from the injectivity of $h_{k+1}^{k}$. It is closed since for $j > l$ the set $h_{j}^{l}(T_l)$ is closed in $T_j$. Let $T_l^*$ denote $q \circ \iota_l(T_l)$. $T_l^*$ is embedded in $T_{l+1}^*$ just as the way $h_{l+1}^{l}(T_l)$ ($= T_{l+1}'$) is embedded in $T_{l+1}$. 
Hence, Figure \ref{3_1knot} can be viewed as a picture of the embedding of $T_l^*$ in $T_{l+1}^{*}$. In general, for $k > l$, $T_l^*$ is embedded in $T_k^*$ just as $h_{k}^{l} (T_l)$ is embedded in $T_k$. 

\begin{figure}[h!]
        \centering
       \includegraphics[ width=8cm, height=10cm]{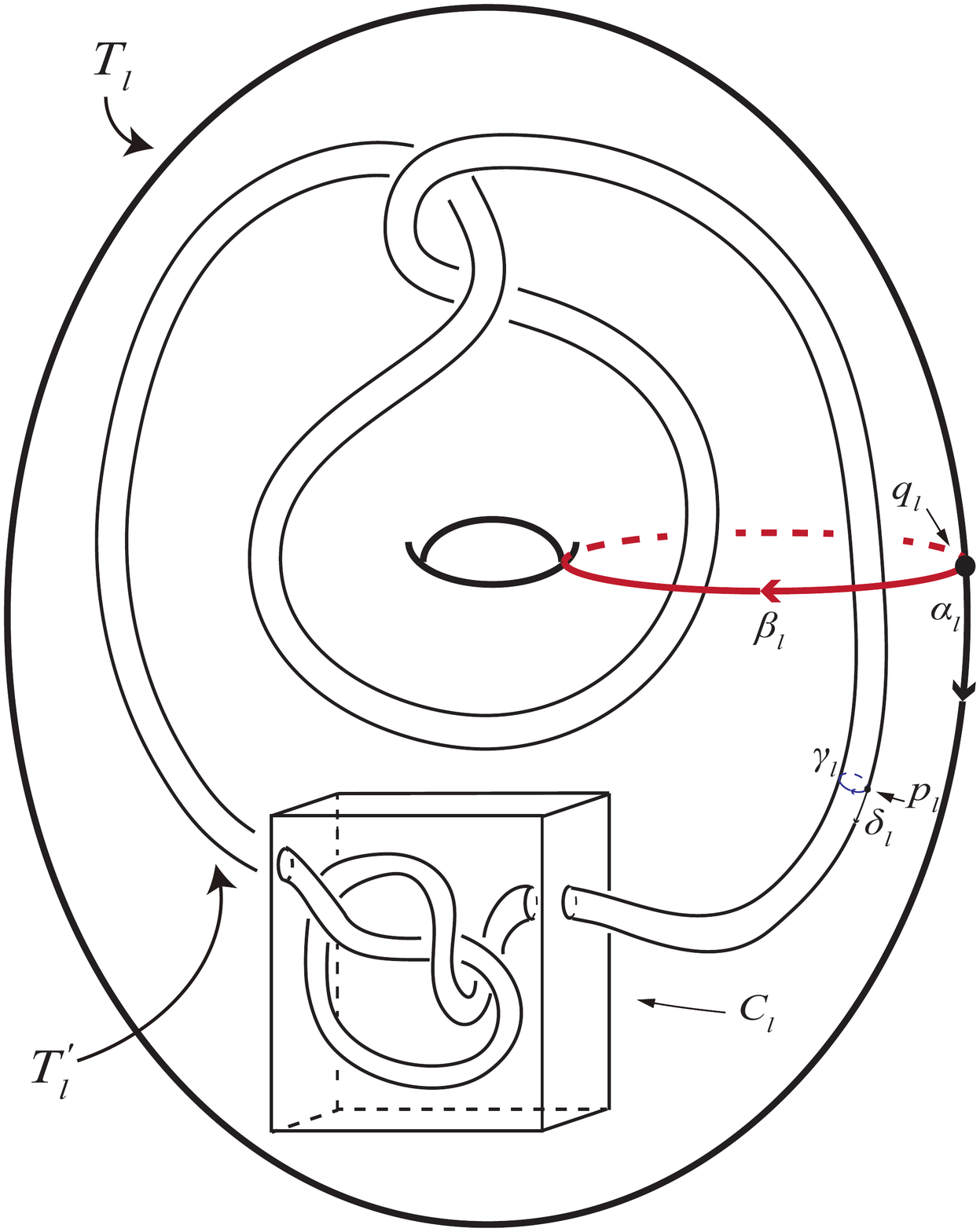}
       \caption{$L_l = T_l \backslash T_l'$. The "inner" boundary component of $L_l$ is $\partial T_l'$. The "outer" boundary component 
       of $L_l$ is $\partial T_l$}
        \label{3_1knot}
\end{figure}

\begin{proposition}\label{Prop: W is contractible}
$W^3$ is an contractible open connected $3$-manifold.
\end{proposition}
\begin{proof}
By the construction described above, $W^3$ is an expanding union of $T_l^*$'s, hence, connected. The interior of each $h_{j}^{l}(T_l)$ is open in $T_j$, so $\operatorname{Int}T_l^*$ is open in $W^3$. Since $T_l^*$ is contained in $\operatorname{Int}T_{l+1}^*$, $W^3$ is an open 3-manifold.

To show the contractibility of $W^3$, we first triangulate $W^3$ by choosing for each $T_l$ ($l\geq 0$), a simplicial subdivision such that each embedding $h_{k+1}^{k}$ ($k \geq 0$) is simplicial with respect to the chosen subdivision of its domain and range. Let $H: W^3 \times [0,1] \to W^3$ be the contraction to be constructed. Define $H$ inductively on the skeleton of $W^3 \times [0,1]$. Pick $p\in W^3$ to be the point to which we want to contract. Map each vertex cross $[0,1]$ to a path beginning at the vertex and ending at $p$. Let $\Delta^{(1)}$ be a 1-simplex of $W^3$. Define the restrictions $H|_{\Delta^{(1)}\times \{0\}}$ to be the identity and $H|_{\Delta^{(1)}\times \{1\}}$ to be the constant map taking all points to $p$. Note that $\partial \Delta^{(1)}$ lies in the 0-skeleta of $W^3$. $H$ has already been defined on $\partial \Delta^{(1)} \times [0,1] = \partial (\Delta^{(1)} \times [0,1])$. Note that $T_l^*$ contracts in $T_{l+1}^*$ (see Figure \ref{3_1knot}). $H$ can be extended to the rest of $\Delta^{(1)} \times [0,1]$ by the fact that $H|_{\partial (\Delta^{(1)}\times [0,1])}$ contracts in $W$. Doing this for all 1-simplexes so $H$ is well-defined on the 1-skeleta cross $[0,1]$. One can do this for 2- and 3-skeleta cross $[0,1]$ inductively.
\end{proof}

\begin{definition}
A topological space $X$ is \emph{locally \emph{1}-connected at the point} $x\in X$ if for each neighborhood $U$ of $x$ there is a neighborhood $V$ of $x$, $V\subset U$, such that every loop in $V$ contracts in $U$. We say that $X$ is \emph{locally \emph{1}-connected} if $X$ is locally 1-connected at each of its points.
\end{definition}

The approach of proving Theorem \ref{Thm: W^3 embeds in no compact ANR} does not rely on Haken's finiteness theorem \cite{Hak68}. Instead, we take advantage of the covering space argument in \cite{Ste77}.

Suppose there is a compact, locally connected, locally 1-connected metric space $U$ such that $U$ contains $W^3$ as an open subset. By taking the component of $U$ containing $W^3$ we may assume that $U$ is connected. Then the following result assures that $\pi_1(U \backslash \operatorname{Int}T^*_0)$ must be finitely generated. 

\begin{lemma}\cite[Lemma 1.1, P.7]{Ste77}
If $X$ is a compact, connected, locally connected, locally $1$-connected metric space, then $\pi_1(X)$ is finitely generated.
\end{lemma}

Instead of working on $\pi_1(U\backslash \operatorname{Int}T^*_0)$ directly, it is easier to focus on a knot space $K_j = S^3 \backslash \operatorname{Int}h_j^0(T_0)$ ($j \geq 1$).\footnote{In \cite{Ste77}, $K_i$ (instead of our $K_j$) denotes the knot space corresponding to his 3-dimensional example $W$. In addition, $K_i$ is homeomorphic to an amalgamation $A_i$ in his thesis. At the end of this section, we also decompose $K_j$ into an amalgamation (see (\ref{amalgamation of K})).} Combining with Claim \ref{Claim}, we have an observation as follows. 

\begin{claim}
 $\pi_1(K_j)$ is a homomorphic image of $\pi_1(U \backslash \operatorname{Int}T^*_0)$.
\end{claim}
\begin{proof}
 Let $p_j$ and $p_j'$ be quotient maps in the commutative diagram (see Figure \ref{Commutative diagram}).
 \begin{figure}[h]
 \begin{center}
 \begin{tikzpicture}[>=angle 90]
 \matrix(a)[matrix of math nodes,
 row sep=3em, column sep=2.5em,
 text height=1.5 ex, text depth=0.25ex]
 {T_j^*\backslash \operatorname{Int}T_0^* & U \backslash \operatorname{Int}T_0^*     \\
  (T_j^*\backslash \operatorname{Int}T_0^*)/\partial T_j^* & (U\backslash \operatorname{Int}T_0^*)/(U \backslash \operatorname{Int}T_j^*) \\};
 \path[->](a-1-1) edge node[above]{$\iota_j$} (a-1-2);
 \path[->](a-2-1) edge node[above]{$g_j$} node[below]{$\approx$} (a-2-2);
 \path[->] (a-1-2) edge node[right]{$p_j$} (a-2-2);
 \path[->] (a-1-1) edge node[right]{$p_j'$} (a-2-1);
 \end{tikzpicture}
 \end{center}
 \caption{Commutative diagram}
 \label{Commutative diagram}
\end{figure}
The inclusion, $\iota_j$, followed by $p_j$ induces the map $g_j$ since the restriction of $p_j$ on $T_j^* \backslash \operatorname{Int}T_{0}^{*}$
is to collapse $\partial T_{j}^{*}$ to a point. It's not hard to see that $g_j$ is actually a homeomorphism. Since $\partial T_{j}^{*}$ is collared in $T_{j}^{*}\backslash \operatorname{Int}T_{0}^{*}$, Lemma \ref{lemma: collar} implies that $p_j'$ induces a surjection on fundamental groups. By the commutativity of the diagram \ref{Commutative diagram}, $p_{j^*}' = g_{j^*}^{-1}p_{j^*}\iota_{j^*}$, where $p_{j^*}'$, $g_{j^*}$, $p_{j^*}$ and $ \iota_{j^*}$ are the homomorphisms induced by maps $p_j'$, $g_j$, $p_j$ and $\iota_j$ respectively. Since $p_{j^*}'$ is a surjection, $g_{j^*}^{-1}p_{j^*}$ is also a surjection. Hence, $\pi_1((T_j\backslash \operatorname{Int}T_0^*)/\partial T_j^*)$ is a homomorphic image of $\pi_1(U\backslash\operatorname{Int}T_0^*)$. According to the construction of $W^3$,
the pair $(T_j^*,T_0^*)$ is homeomorphic to the pair $(T_j,h_j^0 (T_0))$. Then the claim follows from Claim \ref{Claim}. 
\end{proof}
Since the rank\footnote{When we say the \emph{rank} of a group $G$, denoted by $\operatorname{Rank}G$, it means the smallest cardinality of a generating set for $G$.} of a group must be a least as large as that of any homomorphic image, it suffices to show that the rank of $\pi_1(K_j)$ is unbounded.

The space $K_j$ is advertised as "knot space" is because it can be viewed as a knot complement. To see that, we need the construction based on two important tools in producing knots. The first one is

\begin{definition}
Let $K_P$ be a non-trivial knot in $S^3$ and $V_P$ an unknotted solid torus in $S^3$ with $K_P\subset V_P \subset S^3$. Let $K_C \subset S^3$ be another knot and let $V_C$ be a tubular neighborhood of $K_C$ in $S^3$. Let $h: V_P \to V_C$ be a homeomorphism and let $K_W$ be $h(K_P)$. We say $K_C$ is a \emph{companion} of any knot $K_W$ constructed (up to knot type) in this manner. If $h$ is \emph{faithful}, meaning that $h$ takes the preferred longitude\footnote{"Preferred longitude" means that $K_W$ has writhe number zero.} and meridian of $V_P$ respectively to the preferred longitude and meridian of $V_C$,
We say $K_W$ is an \emph{untwisted Whitehead double} of $K_C$. Otherwise, $K_W$ is a \emph{twisted Whitehead double}. For instance, Figure \ref{whiteheaddouble_3_1} is a 3-twisted Whitehead double of a trefoil knot. The pair $(V_P, K_P)$ is the
\emph{pattern} of $K_W$. 
\end{definition}

\begin{figure}[h!]
        \centering
       \includegraphics[ width=10cm, height=6cm]{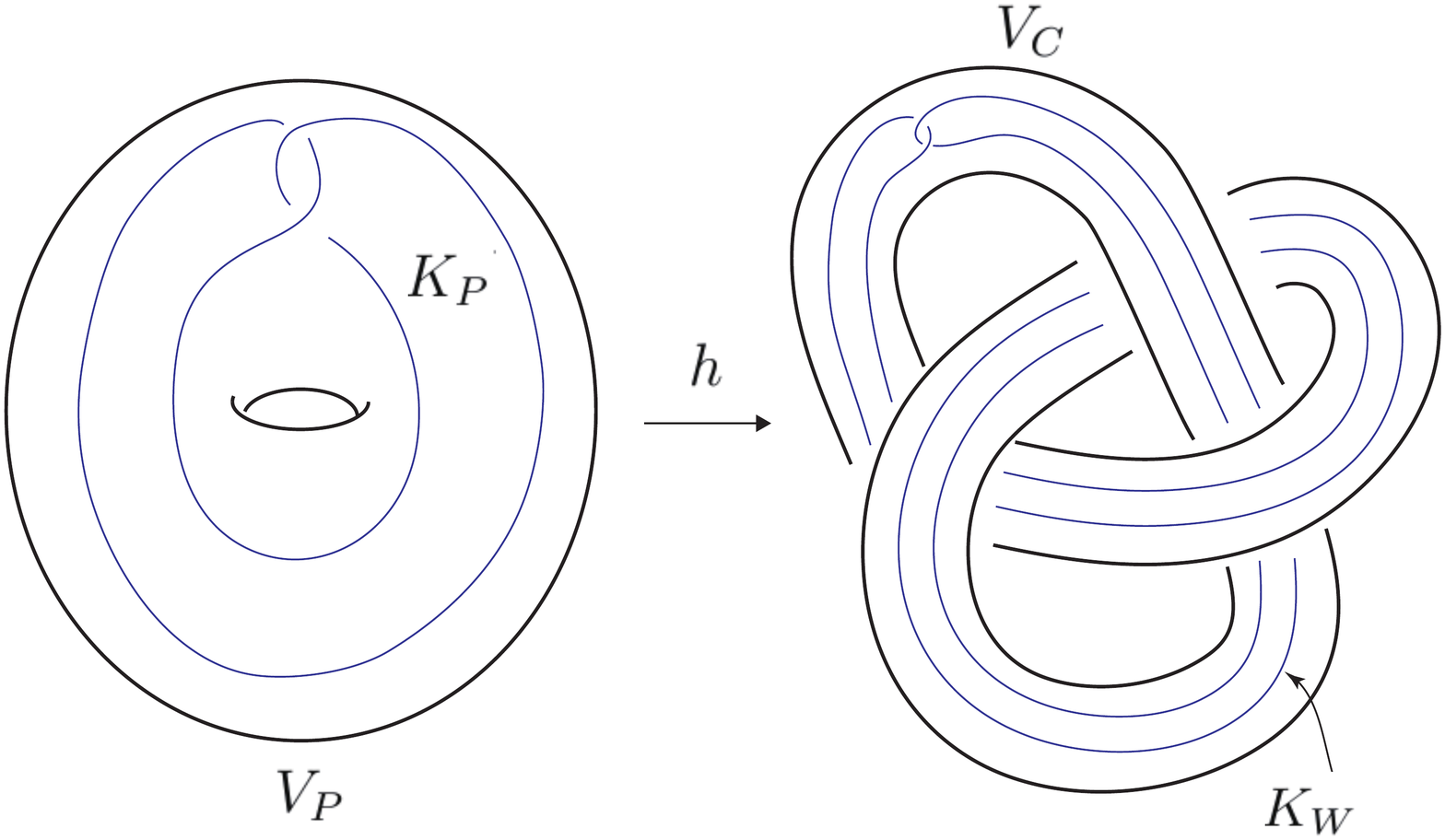}
       \caption{A 3-twisted Whitehead double of a trefoil knot}
        \label{whiteheaddouble_3_1}
\end{figure}

The second tool is based on a type of connected sum of a pair of manifolds $(M_{1}^{m},N_{1}^{n})\#(M_{2}^m,N_{2}^{n})$, where $N_{i}^{n}$ is a locally flat submanifold of $M_{i}^{m}$. Treat the above pair as $(S^3, k_1) \# (S^3,k_2)$ where $k_i$ are tame knots. Removing a standard ball pair $(B_i^3,B_i^1)$ from $(S^3,k_1)$ and gluing the resulting pairs by a homeomorphism $h: (\partial B_2^3,\partial B_2^1) \to (\partial B_1^3,\partial B_1^1)$ to form the pair connected sum. For convenience, we use $k_1 \# k_2$ other than pairs of manfolds. See \cite{Rol76} for details.

To help readers get a better feeling about group $\pi_1(K_j)$, we show that $\pi_1(K_j)$ is
isomorphic to $\pi_1 \left( (T_j \backslash \operatorname{Int} h_j^0 (T_0)) /\partial T_j  \right)$. Geometrically, $K_j$ is the space obtained by sewing the solid torus $S^3 \backslash \operatorname{Int}T_j$ to $T_j \backslash \operatorname{Int}h_{j}^0 (T_0)$ along $\partial T_j$. We decompose $S^3 \backslash \operatorname{Int}T_j$ into two 3-cells $B_1$ and $B_2$, i.e., 
$S^3 \backslash \operatorname{Int}T_j = B_1 \cup B_2$, where $B_1$ is the thickened meridional disk $D$ in $S^3 \backslash \operatorname{Int}T_j$ with $\partial D = \alpha_j$ (see Figure \ref{fig 3}) and $B_2$ is the closure of the complement of $B_1$ in $S^3 \backslash \operatorname{Int}T_j$. Sewing $B_1$ to $T_j \backslash \operatorname{Int}h_{j}^0 (T_0)$ along an annular neighborhood of $\alpha_j$ in $\partial T_j$. By Seifert-van Kampen, the inclusion $T_j \backslash \operatorname{Int}h_{j}^0 (T_0) \hookrightarrow (T_j \backslash \operatorname{Int}h_j^0 (T_0)) \cup B_1$ induces a surjection on fundamental groups whose kernel is the normal closure of the curve $\alpha_j$ in $\pi_1(T_j \backslash \operatorname{Int}h_{j}^0 (T_0))$.

\begin{figure}[h!]
        \centering
       \includegraphics[ width=8cm, height=10cm]{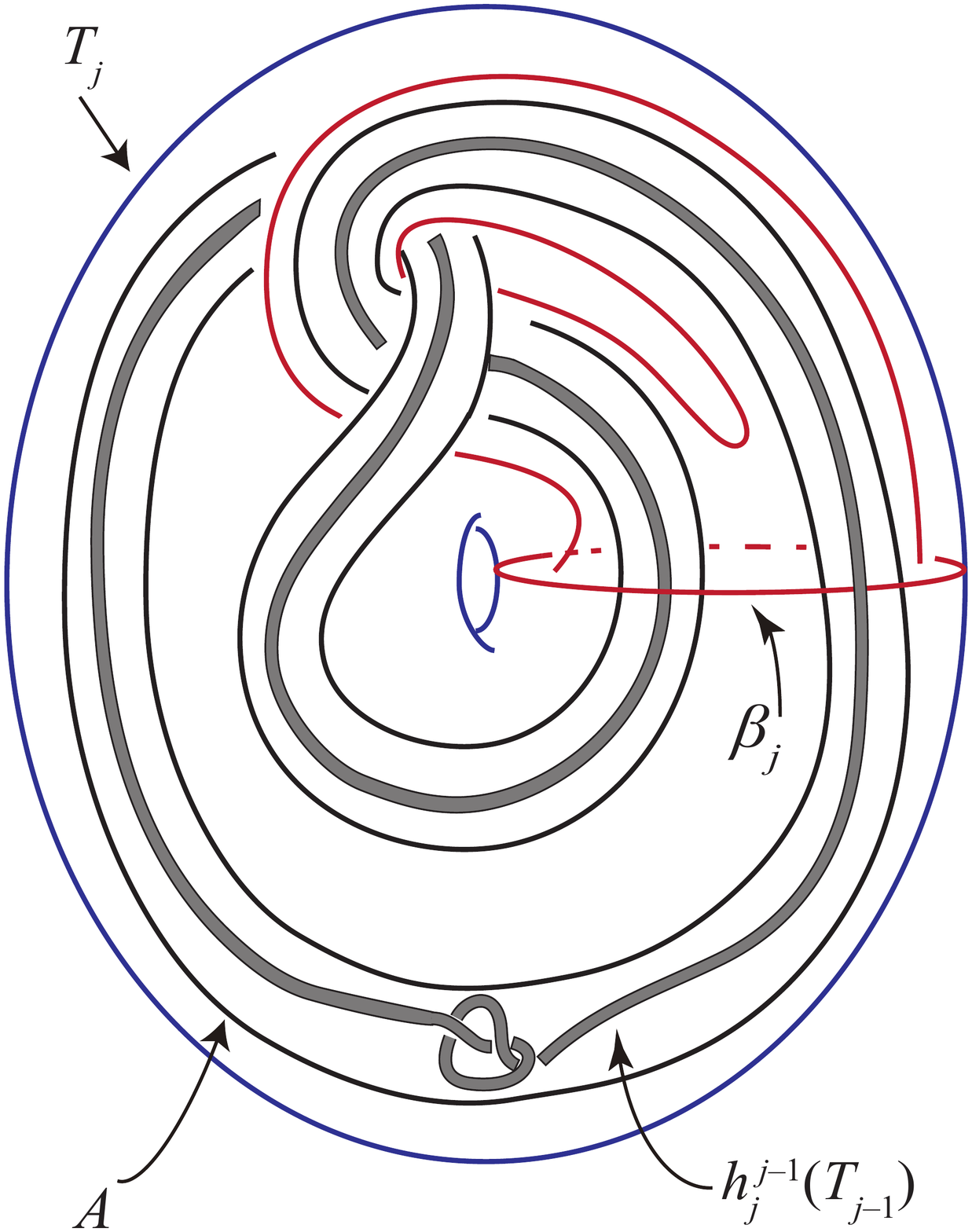}
       \caption{$\beta_j$ contracts in $S_{j-1}^{3}\backslash \operatorname{Int}h_{j}^{0}(T_0)$, where $h_{j}^{0}(T_0)$ is not pictured.}
        \label{fig 3}
\end{figure}

Adding $B_2$ to $(T_j \backslash \operatorname{Int}h_j^0 (T_0)) \cup B_1$ to form the knot complement $K_j$ does not affect the fundamental group. This follows readily from Seifert-van Kampen. Hence, the inclusion $T_j \backslash \operatorname{Int}h_j^0 (T_0) \hookrightarrow K_j$ induces a surjection on fundamental groups whose kernel is the normal closure of the curve $\alpha_j$ in $\pi_1(T_j \backslash \operatorname{Int}h_{j}^0 (T_0))$.

\begin{claim}\label{Claim}
$\pi_1(K_j)$ is isomorphic to $\pi_1 \left( (T_j \backslash \operatorname{Int} h_j^0 (T_0)) /\partial T_j  \right)$.
\end{claim}

\begin{proof}
It's sufficient to show that the meridian $\beta_j$ of $T_j$ is trivial in $\pi_1(K_j)$. In other words, we will show that $\beta_j$ contracts in the complement of $h_j^0(T_0)$. Consider Figure \ref{fig 3}. $h_{j}^0(T_0)$ (not pictured) is contained in $h_{j}^{j-1}(T_{j-1})$, which is also contained in the solid torus $A$. Since $A$ is an unknotted solid torus, $\beta_j$ bounds a 2-chain in $S^3 \backslash A$.
\end{proof}

It's clear that $\pi(K_1)$ is isomorphic to a trefoil knot group.

\begin{claim}
$\pi_1(K_2)$ is isomorphic to the knot group of the connected sum of a trefoil knot and a $3$-twisted Whitehead double of a trefoil knot.
\end{claim}

\begin{proof}
By the construction of $W^3$, $T_{1}^{*}$ embeds in $T_{2}^{*}$ just as the way $T_{0}^{*}$ embeds in $T_{1}^{*}$ (as shown in Figure \ref{3_1knot}). Note that the space
$K_2 = S^3 \backslash \operatorname{Int}h_2^0(T_0)$ can be decomposed into 
$$(S^3 \backslash \operatorname{Int}T_{2}^{*})\cup (T_{2}^{*}\backslash \operatorname{Int}T_{1}^{*})\cup  (T_{1}^{*}\backslash \operatorname{Int}T_{0}^{*}).$$
Since a solid torus $S^3 \backslash \operatorname{Int}T_{2}^{*}$ is glued to $T_{2}^{*}\backslash \operatorname{Int}T_{1}^{*}$ along $\partial T_{2}^{*}$, one can unlink the clasped portion of $T_{1}^{*}$ while keeping the way $T_{0}^{*}$ embeds in $T_{1}^{*}$ via an ambient isotopy $\Psi_t$ of $S^3$ starting at $\Psi_0 = \operatorname{Id}_{S^3}$.  The restriction $\Psi_1(T_{1}^{*})$ is a tubular neighborhood of a trefoil knot, denoted $\mathcal{K}_{*}$. Name a twisted Whitehead double of $\mathcal{K}_{*}$ (as shown in Figure \ref{whiteheaddouble_3_1}) $\mathcal{K}_{*}^{Wh}$. Restrict $\Psi_1$ to 
$T_{0}^{*}$ and deformation retract $\Psi_1(T_{0}^{*})$ onto its core.  The core of $\Psi_1(T_{0}^{*})$ is $\mathcal{K}_{*}^{Wh}$ connected sum with a small trefoil knot, denoted $\mathcal{K}_{**}$. Consider the knot of Figure \ref{double of trefoil}. In this case, $\mathcal{K}_{1}^{Wh}$ is $\mathcal{K}_{*}^{Wh}$ and  $\mathcal{K}_{1}$ is $\mathcal{K}_{**}$. It follows easily that $K_2$ is homotopy equivalent to $S^3 \backslash  (\mathcal{K}_{*}^{Wh}\# \mathcal{K}_{**})$. 
\end{proof}

\begin{figure}[h!]
        \centering
       \includegraphics[ width=9cm, height=8cm]{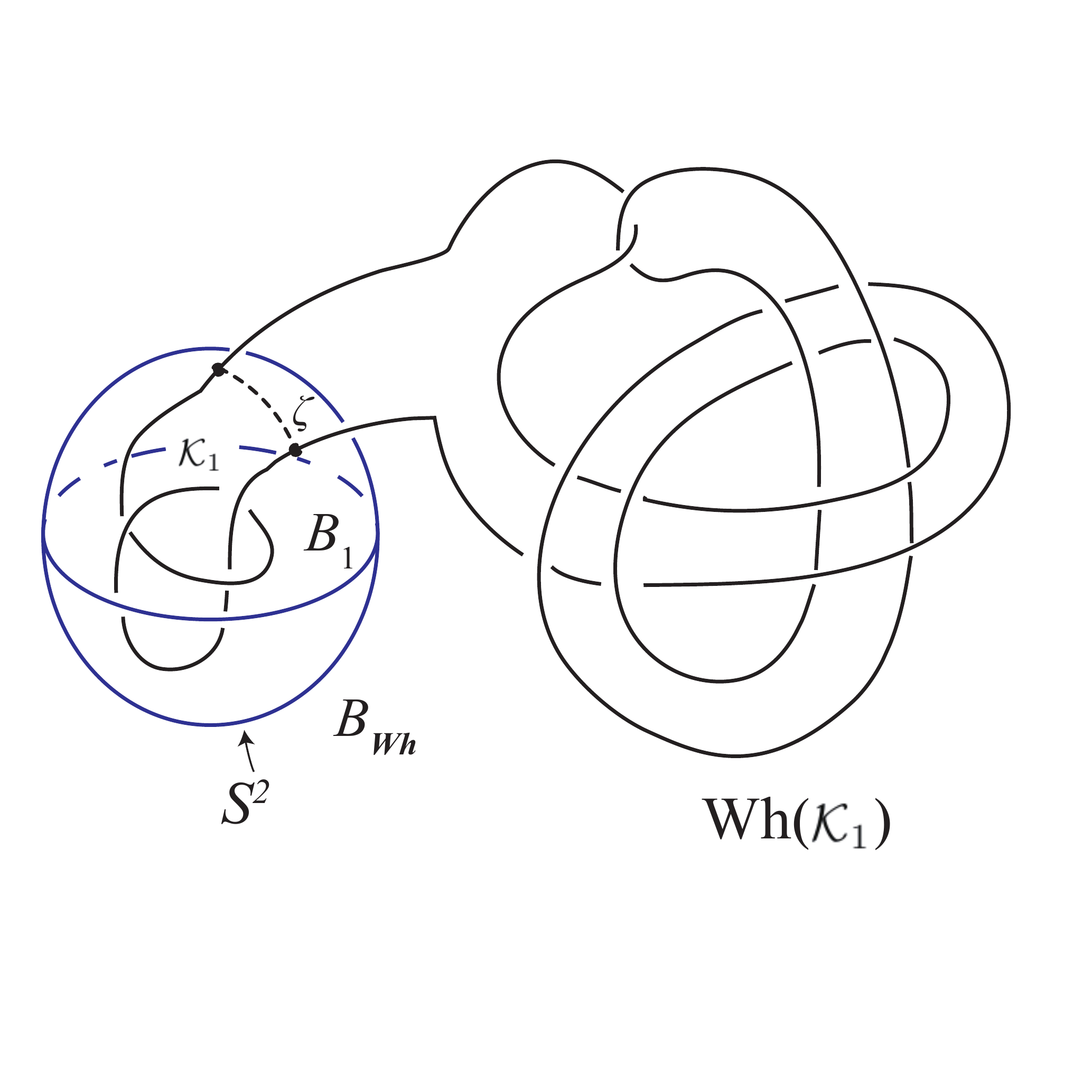}
       \vspace{-2em}
       \caption{The connected sum of a twisted Whitehead double of $\mathcal{K}_1$ and $\mathcal{K}_1 (\approx$ trefoil knot). Here "$\approx$" stands for homeomorphic.}
        \label{double of trefoil}
\end{figure}

Let $\mathcal{K}_1$ be a trefoil knot corresponding to the knot space $K_1$. Denote a knot $\mathcal{K}_2$ by $\mathcal{K}_1^{Wh} \# \mathcal{K}_1$ such that $\pi_1(S^3 \backslash \mathcal{K}_2) \cong \pi_1(K_2)$. Similarly, one can further find a knot $\mathcal{K}_3$ on the 3rd stage which is a connected sum of a twisted Whitehead double of $\mathcal{K}_2$ and $\mathcal{K}_1$. By iteration, a knot $\mathcal{K}_j$ can be viewed as $\mathcal{K}_{j-1}^{Wh}\# \mathcal{K}_1$. 

Let $G_{3_1}$ and $G^{Wh}_{j-1}$ be the knot group of $\mathcal{K}_1$ and $K^{Wh}_{j-1}$ respectively. By the definition of connected sum, there is a tame 2-sphere $S^2$ dividing $S^3$ into two balls $B_{Wh}$ and $B_1$ containing $K^{Wh}_{j-1}$ and $\mathcal{K}_1$ respectively. The intersection of $K^{Wh}_{j-1}$ and $\mathcal{K}_1$ is an arc $\zeta$ lying in $S^2$. View $\mathcal{K}_j = K^{Wh}_{j-1}\#\mathcal{K}_1$ as the union of $K^{Wh}_{j-1}$ and $\mathcal{K}_1$ minus $\operatorname{Int}\zeta$ (see Figure \ref{double of trefoil}). Then we have the following diagram "pushout" commutative diagram \ref{pushout diagram}.

\begin{figure}[h]
\begin{center}
\begin{tikzpicture}[>=angle 90]
\matrix(a)[matrix of math nodes,
row sep=6em, column sep=1em,
text height=1.5 ex, text depth=0.25ex]
{          & \pi_1(S^2 \backslash \mathcal{K}_j) \cong \mathbb{Z}      &           \\
 \pi_1(B_{Wh} \backslash  K^{Wh}_{j-1}) \cong G^{Wh}_{j-1}  &  & \pi_1(B_1 \backslash \mathcal{K}_1) \cong G_{3_1}  \\
                                                           &  \pi_1(S^3 \backslash \mathcal{K}_j)            &                  \\};
\path[right hook->](a-1-2) edge (a-2-1);
\path[right hook->](a-1-2) edge (a-2-3);
\path[right hook->](a-2-1) edge node[right]{$\text{ }\iota_1$} (a-3-2);
\path[right hook->](a-2-3) edge node[left]{$\iota_2$\text{ }} (a-3-2);
\end{tikzpicture}
\end{center}
\caption{"Pushout" commutative diagram}
\label{pushout diagram}
\end{figure}
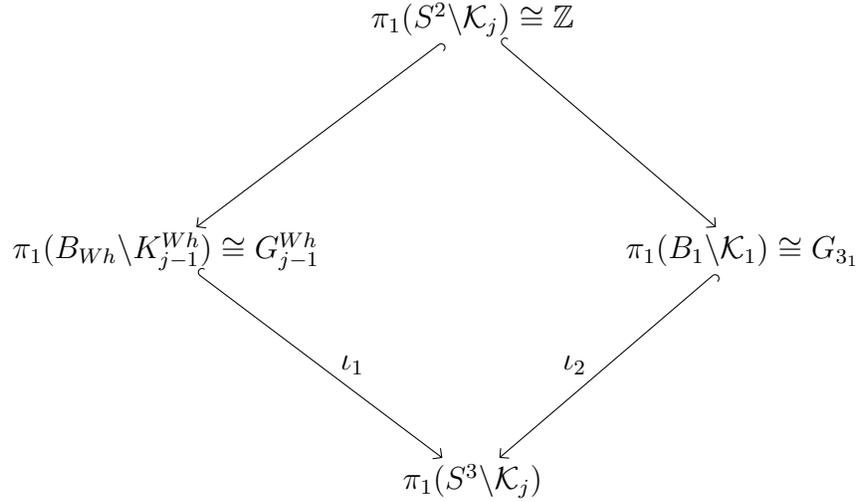
Clearly, the two upper homomorphisms in Figure \ref{pushout diagram} are injective. By the Seifert-van Kampen theorem, the other two homomorphisms $\iota_1,\iota_2$ are also injective. That means 
$$G_{j}= \pi_1(S^3 \backslash \mathcal{K}_j)=G^{Wh}_{j-1} \ast_{\langle \lambda \rangle} G_{3_1}$$ 
is a free product with amalgamation along an infinite cyclic group, where $[\lambda]$ corresponds to the loop class in $\pi_1(S^2 \backslash \mathcal{K}_j)$. According to this set-up, $G^{Wh}_{j-1}$ and $G_{3_1}$ are two subgroups of $G_j$ and $\langle \lambda \rangle$ is a subgroup of both $G^{Wh}_{j-1}$ and $G_{3_1}$. Since both
$G^{Wh}_{j-1}$ and $G_{3_1}$ are abelianized to $\langle \lambda \rangle \cong \mathbb{Z}$, $G_j$ is a split amalgamated free product. Although the work in \cite{Wei99} guarantees a lower bound for $\operatorname{Rank}G_j \ast_{\langle \lambda \rangle} G_{3_1}$, i.e., $\operatorname{Rank}G_j \ast_{\langle \lambda \rangle} G_{3_1} \geq 2$, the ultimate goal is to show that 
$\operatorname{Rank} G_j \ast_{\langle \lambda \rangle} G_{3_1}$ has no upper bound as $j\to \infty$. At the time of writing, we don't know whether there is a direct knot theoretical approach to this. So, we use the covering space theory as developed by Sternfeld in \cite{Ste77}.

We start by constructing a surjective homomorphism $\Phi_j: G^{Wh}_{j-1} \ast_{\langle \lambda \rangle} G_{3_1} \twoheadrightarrow \mathbb{A}_5$, where $\mathbb{A}_5$ is an alternating group on 5 letters. To that end, by the definition of $W^3$,
we decompose $K_j$ into an amalgamation of $L_j$'s. That is, for $j \geq 1$,
\begin{equation}\label{amalgamation of K}
K_j \approx (S^3 \backslash \operatorname{Int}T_j) \cup_{\operatorname{Id}} L_j \cup_{h_{j}^{j-1}} L_{j-1} \cup_{h_{j-1}^{j-2}}\cdots \cup_{h_{2}^{1}} L_1,
\end{equation}
where the sewing homeomorphism $h_{l+1}^{l}$ identifies the boundary component $\partial T_l$ of $L_l$ to the boundary component $\partial T'_{l+1}$
of $L_{l+1}$. It's clear that $\pi_1(K_j)\cong G_j$. So, we convert the problem to finding a surjection from $\pi_1(K_j) \to \mathbb{A}_5$ which will be discussed in the following two sections.

\section{A presentation of $\pi_1(K_j)$}\label{section: A presentation} 
First we spell out a Wirtinger presentation similar to what Sternfeld did in \cite[P.20--26]{Ste77} for $\pi_1(L_l)$, where $l \geq 1$. Let $\Sigma_l$ and 
$\Omega_l$ be polyhedral simple closed curves contained in $S^3$ such that $S^3 \backslash (\Sigma_l \cup \Omega_l)$ deformation retracts onto $L_l$.
$\Sigma_l$ and $\Omega_l$ can be viewed as cores of the solid tori $T_l'$ and $S^3 \backslash \operatorname{Int}T_l$ respecitively (see Figures 
\ref{3_1knot} and \ref{double of 3_1_link}). Let the arc $\mu_l$ in Figure \ref{double of 3_1_link} run from one end point $p_l \in \partial T_l'$ and to the other end point $q_l \in \partial T_l$. $\mu_l$ is properly embedded in $L_l$.

\begin{figure}[h!]
        \centering
       \includegraphics[ width=15cm, height=10cm]{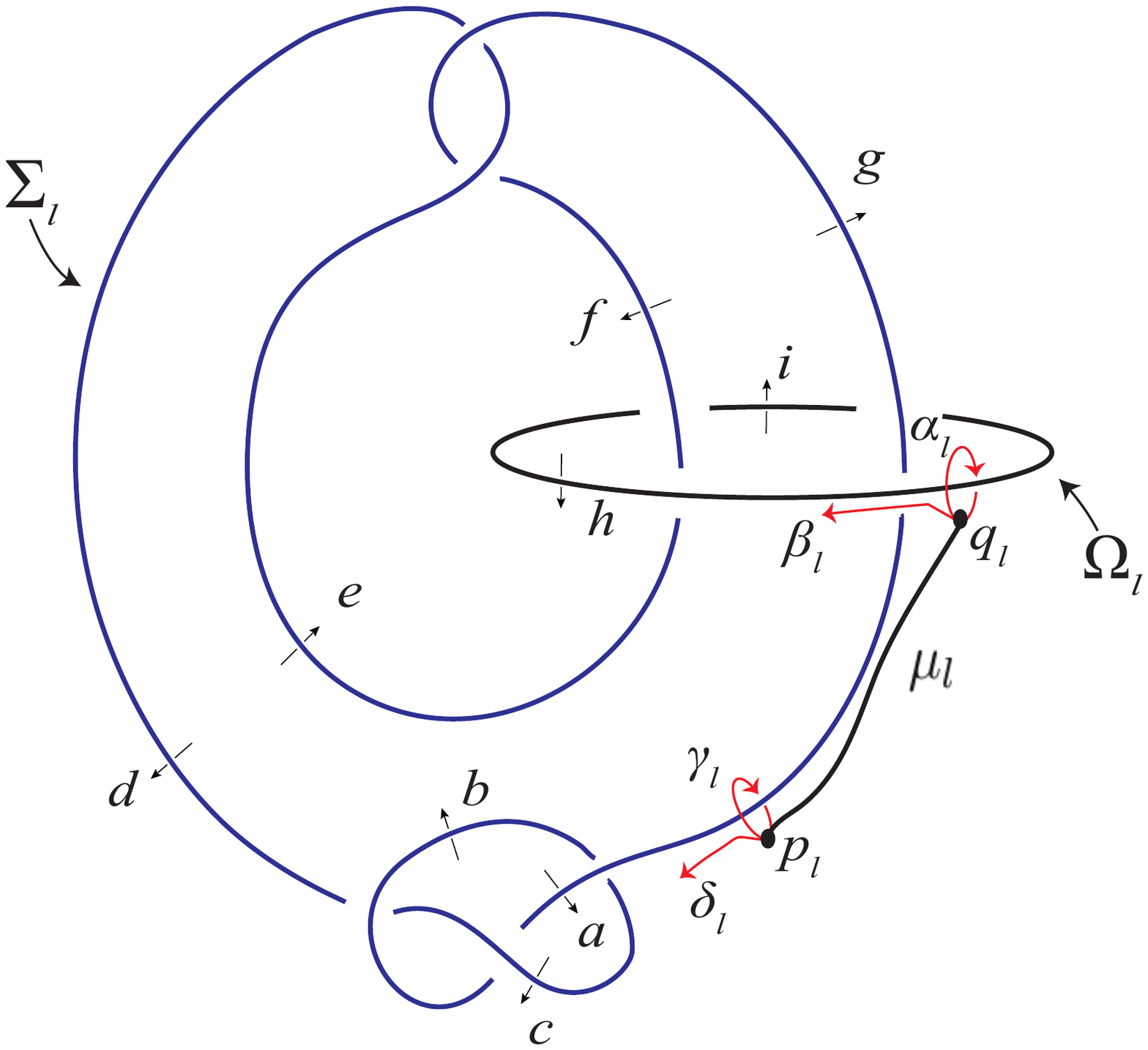}
       \vspace{-2em}
       \caption{A Projection of $\Sigma_l\cup \Omega_l$ into the plane. Arrows $a,b,c,\dots,i$ have subscript $l$ corresponding to $L_l$ ($1<l\leq j$) is suppressed.}
        \label{double of 3_1_link}
\end{figure}

Hence, the presentation of  $\pi_1(S^3 \backslash (\Sigma_l \cup \Omega_l),p_l)$ is 
\begin{equation}\label{group presentation1}
\text{Generators: } a,b,c,\dots, i
\end{equation}
\begin{equation*}
\text{Relators:} \begin{cases}
\begin{aligned} 
R_{l,1}: b &= c^{-1}ac  \\
R_{l,2}: c &= a^{-1}ba  \\
R_{l,3}: d &= b^{-1}cb  \\
R_{l,4}: e &= gdg^{-1}  \\
R_{l,5}: f &= heh^{-1}  \\
R_{l,6}: g &= efe^{-1}  \\
R_{l,7}: a &= h^{-1}gh  \\
R_{l,8}: h &= g^{-1}ig  \\
R_{l,9}: i &= fhf^{-1},
\end{aligned}
\end{cases}
\end{equation*}
where the subscripts $l$'s are surpressed.

Write loop classes $[\alpha_l], [\beta_l], [\gamma_l]$ and $[\delta_l]$ as words in the generators $a_l,b_l,\dots,i_l$ of (\ref{group presentation1}):
\begin{equation}\label{words}
\begin{cases}
\begin{aligned} 
&[\alpha_l] = h_l  \\
&[\beta_l] =  f_{l}^{-1}g_{l}\\
&[\gamma_l] = a_l \\ 
&[\delta_l] = c_la_lb_lg_{l}^{-1}h_{l}^{-1}e_{l}^{-1}h_l \\
\end{aligned}
\end{cases}
\end{equation}
where $[\alpha_l]$ is determined by the oriented simple closed curve $\alpha_l$ lying in $\partial L_l$ (see Figures
\ref{3_1knot} and \ref{double of 3_1_link}) and the arc $\mu_l$ connecting $\alpha_l$ to the base point $p_l$. Likewise, $[\beta_l]$,
$[\gamma_l]$ and $[\delta_l]$ are defined in the same manner. Deformation retract $S^3\backslash(\Sigma_l \cup \Omega_l)$ onto $L_l$. It's clear that Presentation (\ref{group presentation1}) is a presentation of $\pi_1(L_l,p_l)$. Consider the loop classes $a_l,b_l,\dots, i_l$ in 
$\pi_1(L_l,p_l)$ (represented by the same loops as before) as loops in $L_l$. At the same time, $[\alpha_l], [\beta_l], [\gamma_l]$ and $[\delta_l]$ may be written as the same words (\ref{words}) in the generators of $\pi_1(L_l,p_l)$.

Recall in the previous section, we have the following knot space
\begin{equation*}
K_j \approx (S^3 \backslash \operatorname{Int}T_j) \cup_{\operatorname{Id}} L_j \cup_{h_{j}^{j-1}} L_{j-1} \cup_{h_{j-1}^{j-2}}\cdots \cup_{h_{2}^{1}} L_1,
\end{equation*}
where the sewing homeomorphism $h_{l+1}^{l}$ identifies the boundary component $\partial T_l$ of $L_l$ to the boundary component $\partial T'_{l+1}$
of $L_{l+1}$ such that the transverse oriented simple closed curves $\alpha_l$ and $\beta_l$ of $\partial T_l$ are mapped
in an orientation preserving manner to the transverse oriented simple closed curves $\delta_{l+1}$ and $\gamma_{l+1}$ respectively in $\partial T'_{l+1}$. Using the words (\ref{words}), this can be described by the following relators
\begin{equation}\label{relators}
\text{Relators} \begin{cases}
\begin{aligned} 
&S_{l,1}: h_{l-1} = c_la_lb_lg_{l}^{-1}h_{l}^{-1}e_{l}^{-1}h_l \text{ for } j\geq l \geq 2 \\
&S_{l,2}: f_{l-1}^{-1}g_{l-1} = a_l \text{ for } j \geq l \geq 2. \\
\end{aligned}
\end{cases}
\end{equation}

Combine the words (\ref{group presentation1}) and (\ref{relators}), we obtain
\begin{proposition}\label{Prop: presentation}
$\pi_1(K_j,p_1)$, $j\geq 1$, has the following presentation
\begin{equation}\label{group_presentation2}
\text{Generators: } a_l,b_l,c_l,\dots, i_l \text{ for } j\geq l \geq 1
\end{equation}
\begin{equation*}
\text{Relators:} \begin{cases}
\begin{aligned} 
&R_{l,k} \text{ for } j\geq l \geq 1 \text{ and } 9 \geq k \geq 1  \\
&S_{l,1} \text{ for } j\geq l \geq 2 \\
&S_{l,2} \text{ for } j \geq l \geq 2 \\
&h_j = 1,
\end{aligned}
\end{cases}
\end{equation*}
where the generators $a_l,\dots,i_l$ of Presentation (\ref{group_presentation2})  correspond to those of Presentation (\ref{group presentation1}) conjugated by the path $\mu_l$.
\end{proposition}
\begin{proof}
The proof is an easy modification of the proof of Proposition 4.1 in \cite{Ste77}.
\end{proof}

\section{The surjection of $\pi_1(K_j,p_1)$ onto $\mathbb{A}_5$}\label{section: The surjection}

Here we shall define a homomorphism $\Phi_j: \pi_1(K_j,p_1) \to \mathbb{A}_5$, where $j\geq 1$. It suffices to define $\Phi_j$ on 
the generators of Presentation (\ref{group_presentation2}) of $\pi_1(K_j,p_1)$ and check that the definition is compatible with the relators of the presentation. That is, if the following words
$$w(a_1,b_1,\dots,i_1,\dots,a_j,b_j,\dots,i_j) = w'(a_1,b_1,\dots,i_1,\dots,a_j,b_j,\dots,i_j)$$
is a relator of the presentation, then
$$w(\Phi_1(a_1),\dots,\Phi_1(i_1),\dots, \Phi_j(a_j),\dots,\Phi_j(i_j))= w'(\Phi_1(a_1),\dots,\Phi_1(i_1),\dots, \Phi_j(a_j),\dots,\Phi_j(i_j)) $$
must hold for $\mathbb{A}_5$.

Consider an extreme case by "unknotting" every small trefoil knot in the link (corresponding to $L_l$) as shown in Figure \ref{double of 3_1_link}. The link in Figure \ref{double of 3_1_link} can be viewed as a connected sum of a Whitehead link and a trefoil knot. Thus, we can abelianize the trefoil knot group to $\langle a_l \rangle$ while keeping the remaining structure of the group of the link complement fixed. Inherit the definitions of $T_l$, $T'_{l}$ and $h_{l+1}^{l}$ in constructing the knot space $K_j$. Unknot the trefoil-knotted hole in $C_l$ when we embed $T'_l$ into $T_l$. For convenience, we still call such reembedded torus $T'_l$. Similar to the construction of the knot space $K_j$, the sewing homeomorphism $h_{l+1}^{l}$ identifies the boundary component $\partial T_l$ of $L_l$ to the boundary component $\partial T'_{l+1}$ of $L_{l+1}$ such that the transverse oriented simple closed curves $\alpha_l$ and $\beta_l$ of $\partial T_l$ are mapped in an orientation preserving manner to the transverse oriented simple closed curves $\delta_{l+1}$ and $\gamma_{l+1}$ respectively in $\partial T'_{l+1}$. Furthermore, when $\alpha_l$ of $\partial T_l$ is mapped to $\delta_{l+1}$ in $\partial T'_{l+1}$, $h_{l+1}^{l}$ first gives 3 compensating half-twists to $T_l$ due to the writhe of trefoil knot (before abelianization) in $T'_{l+1}$ is 3. In other words, the new knot space is a concatenation of Whitehead links with 3 half-twists. Denote the corresponding knot space by $K_{j}^{**}$. By the above procedure, $\pi_1(K_{j}^{**})$ can be obtained by adding relators $a_l = b_l$, $b_l = c_l$ and $c_l = d_l$ to the presentation of $\pi_1(K_j)$ in Proposition \ref{Prop: presentation}
\begin{equation}\label{group_presen}
\text{Generators: } a_l,b_l,c_l,\dots, i_l \text{ for } j\geq l \geq 1
\end{equation}
\begin{equation*}
\text{Relators:} \begin{cases}
\begin{aligned}
& a_l = b_l, b_l = c_l, c_l = d_l \\ 
&R_{l,k} \text{ for } j\geq l \geq 1 \text{ and } 9 \geq k \geq 1  \\
&S_{l,1} \text{ for } j\geq l \geq 2 \\
&S_{l,2} \text{ for } j \geq l \geq 2 \\
&h_j = 1.
\end{aligned}
\end{cases}
\end{equation*}
Clearly, there is a surjection of $\psi_j: \pi_1(K_j) \twoheadrightarrow \pi_1(K_j^{**})$ by sending $a_l,\dots, d_l$
in Presentation (\ref{group_presentation2}) to $a_l$ in Presentation (\ref{group_presen}). So, it suffices to find a surjection $\phi_j$
of $\pi_1(K_j^{**})$ onto $\mathbb{A}_5$.

We shall define $\phi_j$ inductively on the generators of Presentation (\ref{group_presen}). If $j=1$, we use  GAP \cite{GAP18} to define a surjection $\phi_1$ on $a_1,\dots, i_1$ by Table \ref{Table 1}. This definition is compatible with the relators $R_{1,k}$ and $h_1 = 1$, where $1\leq k\leq 9$. If $j=2$, both Tables \ref{Table 1} and \ref{Table 2} are used. Besides relators $R_{1,k}, R_{2,k}$ and $h_2 =1$, relators $S_{2,1}$ and $S_{2,2}$ are also compatible. Similarly, if $j =3$ (resp. $j=4$), Tables \ref{Table 1}-\ref{Table 3} (resp. \ref{Table 1}-\ref{Table 4}) are applied. When $j\geq 5$, Tables \ref{Table 1}-\ref{Table 5} will be applied periodically. That is, extend $\phi_j$ to the generators $a_l,\dots,i_l$
according to Table \ref{Table 1} if $l = j$, according to Table \ref{Table 2} if $l = j - 1 - 4T$, according to Table \ref{Table 3} if
$l = j - 2 - 4T$, according to Table \ref{Table 4} if $l = j - 3 - 4T$ and according to Table \ref{Table 5} if $l = j - 4 - 4T$, where
$T \in \mathbb{N}$ and $0 \leq T \leq (j - 1)/4$. One can either use GAP \cite{GAP18} or simply by hand to check such extension is compatible with relators in Presentation (\ref{group_presentation2}). Hence, the composition $\Phi_j = \phi_j \circ \psi_j$ is the desired surjection.

\begin{table}[h]
\centering
\caption{}
\subfloat[$l= j$]{
\begin{tabular}{|c|c|}
\hline
Generators & Image   \\ \hline
$a_l$ & (1,2)(3,4)\\
$b_l$ & (1,2)(3,4)\\
$c_l$ & (1,2)(3,4)\\
$d_l$ & (1,2)(3,4)\\
$e_l$ & (1,2)(3,4)\\
$f_l$ & (1,2)(3,4)\\
$g_l$ & (1,2)(3,4)\\
$h_l$ & ()\\
$i_l$ & ()\\\hline
\end{tabular}
\label{Table 1}}
\subfloat[$l= j - 1 - 4T$]{
\begin{tabular}{|c|c|}
\hline
Generators & Image   \\ \hline
$a_l$ & (1,2,3)\\
$b_l$ & (1,2,3)\\
$c_l$ & (1,2,3)\\
$d_l$ & (1,2,3)\\
$e_l$ & (2,4,3)\\
$f_l$ & (1,3,4)\\
$g_l$ & (1,4,2)\\
$h_l$ & (1,2)(3,4)\\
$i_l$ & (1,3)(2,4)\\ \hline
\end{tabular}
\label{Table 2}}
\subfloat[$l= j - 2 - 4T$]{
\begin{tabular}{|c|c|}
\hline
Generators & Image   \\ \hline
$a_l$ & (1,3)(4,5) \\
$b_l$ & (1,3)(4,5)\\
$c_l$ & (1,3)(4,5)\\
$d_l$ & (1,3)(4,5)\\
$e_l$ & (1,2)(4,5)\\
$f_l$ & (1,3)(4,5)\\
$g_l$ & (2,3)(4,5)\\
$h_l$ & (1,2,3)\\
$i_l$ & (1,3,2)\\ \hline
\end{tabular}
\label{Table 3}}
\end{table}

\begin{table}[h]
\centering
\caption{}
\subfloat[$l = j- 3 - 4T$]{
\begin{tabular}{|c|c|}
\hline
Generators & Image   \\ \hline
$a_l$ & (3,4,5) \\
$b_l$ &  (3,4,5)\\
$c_l$ & (3,4,5)\\
$d_l$ & (3,4,5)\\
$e_l$ & (1,3,5)\\
$f_l$ & (1,4,3)\\
$g_l$ & (1,5,4)\\
$h_l$ & (1,3)(4,5)\\
$i_l$ & (1,5)(3,4)\\
\hline
\end{tabular}
\label{Table 4}}
\subfloat[$l=  j- 4 - 4T$]{
\begin{tabular}{|c|c|}
\hline
Generators & Image   \\ \hline
$a_l$ &  (1,2)(3,4)\\
$b_l$ & (1,2)(3,4)\\
$c_l$ & (1,2)(3,4)\\
$d_l$ & (1,2)(3,4)\\
$e_l$ & (1,2)(4,5)\\
$f_l$ & (1,2)(3,4)\\
$g_l$ & (1,2)(3,5)\\
$h_l$ & (3,4,5)\\
$i_l$ & (3,5,4)\\\hline
\end{tabular}
\label{Table 5}}
\end{table}

\begin{remark} \label{Error}
In line 16 \cite[P.28]{Ste77}, the author claims that the definition of $\Phi_i: \pi_1(A_i) \to A$ given in Table 1 \cite[P.29]{Ste77} is compatible with the relators $S_{j,1},S_{j,2}$ for $l\geq j \geq 2$, where $A$ is an alternating group on 5 letters $v$, $w$, $x$, $y$ and $z$. However, for $l <i$, $\Phi(o_{l-1}^{-1}h_{l-1}f_{l-1}^{-1}q_{l-1})$ is not equal to $\Phi(a_{l})$. That is,
using Table 1 \cite[P.29]{Ste77}, $\Phi(o_{l-1}) = (vy)(wz)$, $\Phi(h_{l-1})=(vy)(xz)$, $\Phi(f_{l-1}) = (wx)(yz)$ and $\Phi(q_{l-1}) = (vw)(yz)$. Hence, $\Phi(o_{l-1}^{-1}h_{l-1}f_{l-1}^{-1}q_{l-1})=(vw)(xz)=\Phi(r_{l})\neq \Phi(a_{l}) =(vw)(xy)$. That means the definition of the so claimed $\Phi_i$ is not compatible with the relators $S_{j,1},S_{j,2}$ for $l\geq j \geq 2$. This error directly affects the following statement \cite[P.52]{Ste77}:
"The composition $\pi_1(C_j,x_j)\xrightarrow{k_\ast} \pi_1(A_i,x_j)\xrightarrow{M_j} \pi_1(A_i,x_i) \xrightarrow{\Phi_i} A$ has image isomorphic 
to $\mathbb{Z}_2$ in $A$ since $\Phi_i$ maps $a_j$ and $b_j$ to the same element of order 2 in $A$. Thus, the kernel of $\Phi_i \circ M_j \circ k_\ast$ has index 2 in $\pi_1(C_j,x_j)$." To fix this error, we provide a series of correct tables here.

We have to use at least 3 tables (instead of 2 tables) such that the definition of $\Phi_i$ is compatible with all the relators. Similar to how we define a surjection of $\pi_1(K_j,p_1) \to \mathbb{A}_5$ in the beginning of this section, with the assistance of GAP \cite{GAP18}, the following tables provide a surjection of $\Phi_i: \pi_1(A_i,x_1) \twoheadrightarrow \mathbb{A}_5$. If $i = 1$, we defined $\Phi_i$ on $a_1,\dots,u_1$
by Table \ref{Table 6}. If $i=2$, then Tables \ref{Table 6} and \ref{Table 7} are used. Otherwise, when $ i\geq 3$, Tables \ref{Table 6}, \ref{Table 7} and \ref{Table 8} are applied. That is, extend $\Phi_i$ to the generators $a_l,\dots,u_l$ according to Table \ref{Table 6} if $l=i$, according to Table \ref{Table 7} at $l = i-1 - 2T$ and according to Table \ref{Table 8} at $l = i-2 - 2T$, where $T \in \mathbb{N}$ and $0 \leq T \leq (i - 1)/2$.

\begin{table}
\caption{}
\subfloat[$l= i$]{
\begin{tabular}{|c|c|}
\hline
Generators & Image   \\ \hline
$a_l$ &  (1,2)(3,5) \\
$b_l$ & (1,2)(3,5) \\
$c_l$ & (1,2)(3,5)\\
$d_l$ & (1,2)(3,5)\\
$e_l$ & (1,2)(4,5)\\
$f_l$ & (1,2)(4,5)\\
$g_l$ & (1,2)(4,5)\\
$h_l$ & (1,2)(3,5)\\
$i_l$ & (1,2)(3,5)\\
$j_l$ & (1,2)(4,5)\\
$k_l$ & (1,2)(3,4)\\
$l_l$ & (1,2)(4,5)\\
$m_l$ & (1,2)(3,4)\\
$n_l$ & (1,2)(3,4)\\
$o_l$ & (1,2)(3,4)\\
$p_l$ & (1,2)(3,4)\\
$q_l$ & (1,2)(3,5)\\
$r_l$ & ()\\
$s_l$ & ()\\
$t_l$ & ()\\
$u_l$ & ()\\ \hline
\end{tabular}
\label{Table 6}}
\subfloat[$l = i-1 -2T$]{
\begin{tabular}{|c|c|}
\hline
Generators & Image   \\ \hline
$a_l$ & (1,2)(4,5)\\
$b_l$ & (1,2)(4,5)\\
$c_l$ & (1,2)(4,5)\\
$d_l$ & (1,2)(4,5)\\
$e_l$ & (1,3)(4,5)\\
$f_l$ & (2,5)(3,4)\\
$g_l$ & (1,5)(2,4)\\
$h_l$ & (1,4)(3,5)\\
$i_l$ & (2,4)(3,5)\\
$j_l$ & (1,3)(2,5)\\
$k_l$ & (2,3)(4,5)\\
$l_l$ & (1,3)(4,5)\\
$m_l$ & (1,3)(4,5)\\
$n_l$ & (1,5)(2,4)\\
$o_l$ & (1,4)(2,3)\\
$p_l$ & (1,5)(2,3)\\
$q_l$ & (1,2)(3,4)\\
$r_l$ & (1,2)(3,5)\\
$s_l$ & (1,2)(4,5)\\
$t_l$ & (1,5)(2,3)\\
$u_l$ & (2,5)(3,4)\\
\hline
\end{tabular}
\label{Table 7}}
\subfloat[$l= i - 2 - 2T$]{
\begin{tabular}{|c|c|}
\hline
Generators & Image   \\ \hline
$a_l$ & (1,2)(3,5)\\
$b_l$ & (1,2)(3,5)   \\
$c_l$ & (1,2)(3,5)  \\
$d_l$ & (1,2)(3,5)  \\
$e_l$ & (1,4)(3,5)  \\
$f_l$ & (2,5)(3,4)   \\
$g_l$ & (1,5)(2,3)   \\
$h_l$ & (1,3)(4,5) \\
$i_l$ & (2,3)(4,5)  \\
$j_l$ & (1,4)(2,5) \\
$k_l$ & (2,4)(3,5) \\
$l_l$ & (1,4)(3,5)\\
$m_l$ & (1,4)(3,5)\\
$n_l$ & (1,5)(2,3)\\
$o_l$ & (1,3)(2,4)\\
$p_l$ & (1,5)(2,4)\\
$q_l$ & (1,2)(3,4)\\
$r_l$ & (1,2)(4,5)\\
$s_l$ & (1,2)(3,5)\\
$t_l$ & (1,5)(2,4)\\
$u_l$ & (2,5)(3,4)\\\hline
\end{tabular}
\label{Table 8}}
\end{table}
\end{remark}

\section{Properties of a cube with a trefoil-knotted hole} \label{section: properties of cube hole}

One of the key ingredients in proving Theorem \ref{Thm: W^3 embeds in no compact ANR} is to understand the covering space of a cube with a trefoil-knotted hole as shown in Figure \ref{3_1knot}. In this section, we collect a number of important properties about cubes with a trefoil-knotted hole. Let $C$ be the cube with a trefoil-knotted hole as shown in Figure \ref{cubehole_3_1}. Here $C$ is the complement in $S^3$ of the interior of a regular neighborhood of the polyhedral simple closed curve $\Gamma$. There is a deformation retract of $S^3\backslash \Gamma$ onto $C$. The presentation of $\pi_1(S^3 \backslash \Gamma)$ (i.e., trefoil knot group) is a presentation of $\pi_1(C,p_0)$, where $p_0$ is a base point. Hence, one can use the Wirtinger presentation of $\pi_1(S^3 \backslash \Gamma)$ to obtain the following proposition.

\begin{figure}[h!]
        \centering
       \includegraphics[ width=9cm, height=12cm]{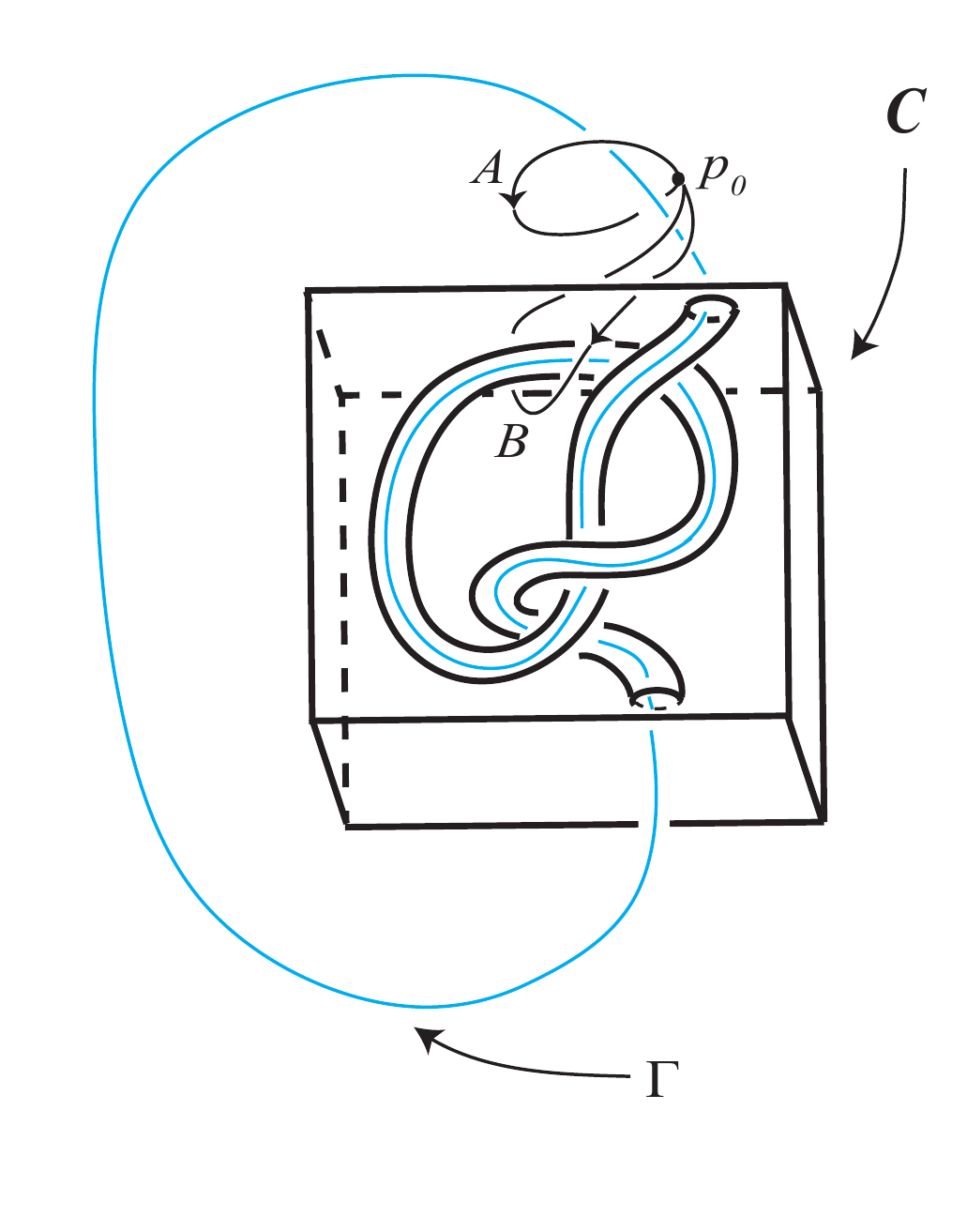}
       \caption{The Cube-With-trefoil-Knotted Hole. $C$ is the complement in $S^3$ of the interior of a regular neighborhood of the trefoil-knotted simple closed curve $\Gamma$}
        \label{cubehole_3_1}
\end{figure}

\begin{proposition}\label{Prop: presentation of 3_1 knot}
$\pi_1(C,p_0)$ has presentation 
$$\langle a, b|
b^{-1}a^{-1}b^{-1}aba = 1 \rangle,$$
where $a = [A]$ and $b = [B]$ as shown in Figure \ref{cubehole_3_1}.
\end{proposition}

\begin{corollary} \label{Corollary: rank of pi_1(C)}
$\pi_1(C,p_0)$ has $\operatorname{Rank}2$.
\end{corollary}

\begin{proof}
Obviously, $\operatorname{Rank} \pi_1(C,p_0)\leq 2$. By the classification of finite simple groups, $\operatorname{Rank} \mathbb{A}_5 = 2$. Using GAP \cite{GAP18}, one can find a surjection of $\pi_1(C,p_0)$ onto $\mathbb{A}_5$ by $(a,b) \mapsto \big((1,3,5,4,2),(1,2,3,4,5)\big) $. That means $\operatorname{Rank} \pi_1(C,p_0)$ has to be greater or equal to 2. Hence, $\operatorname{Rank} \pi_1(C,p_0)= 2$.
\end{proof}

\begin{proposition}\cite[Prop.6.3]{Ste77}\label{Prop: 2-fold cover}
$C$ has a unique $2$-fold cover, $\tilde{C}^2$, the boundary $\partial \tilde{C}^2$ is connected and the quotient map 
$$ Q: \tilde{C}^2 \to \tilde{C}^2/\partial \tilde{C}^2$$
induces a surjection on fundamental groups.
\end{proposition}

\begin{lemma}\cite[Lemma 1.3]{Ste77} \label{lemma: collar}
Let $B$ be a subspace of $X$. Let $B$ and $X$ be path connected. If $B$ is collared in $X$, then the quotient map $q: X \to X/B$ 
induces a surjection of fundamental groups whose kernel is the normal closure in $\pi_1(X)$ of $i_\ast \pi_1(B)$, where $i_*$ denotes the inclusion induced homomorphism.
\end{lemma}

The following result generalizes Proposition \ref{Prop: 2-fold cover} for the $k$-fold cyclic cover of $C$.
\begin{proposition}\label{Prop: k-fold cover}
Let $\tilde{C}^k$ be the $k$-fold cyclic cover of $C$. Then $\partial \tilde{C}^k$ is connected and the quotient map
$$ Q: \tilde{C}^k \to \tilde{C}^k/\partial \tilde{C}^k$$
induces a surjection on fundamental groups.
\end{proposition}
\begin{proof}
First, we show $\partial \tilde{C}^k$ is connected. Let $P: \tilde{C}^k \to C$ be the $k$-fold cyclic cover. The restriction of $P$ to each component of $P^{-1}(\partial C)$ is a covering map of $\partial C$. Note that the $k$-fold cyclic cover is defined to be the one which corresponds to the kernel of the composite 
$$\pi_1(C)\xrightarrow{abelianization}\mathbb{Z}\xrightarrow{projection}\mathbb{Z}_k.$$
The uniqueness of the abelianization and the projection assures that the simple closed curve $A$ (see Figure \ref{cubehole_3_1}) in $\partial C$ based at a point $p_0$ has a lift $\tilde{A}$ which is not a loop since the loop $[A]$ corresponding to the generator $a$ in Proposition \ref{Prop: presentation of 3_1 knot} is not in the kernel. Therefore, the component of $\partial \tilde{C}^k$ that contains $\tilde{A}$ must be a least 
a double cover of $\partial C$ since the two end points of $\tilde{A}$ cover $p_0$. Since each point of $C$ has precisely $k$ preimages in $\tilde{C}^k$, the component of $\partial \tilde{C}^k$ that contains $\tilde{A}$ must be all of $\partial \tilde{C}^k$. Thus $\partial \tilde{C}^k$ is (path) connected.

Applying Lemma \ref{lemma: collar} finishes the proof.
\end{proof}

\begin{proposition} \label{Prop: Rank of 2-fold cover}
$\pi_1(\tilde{C}_2/\partial \tilde{C}_2) \cong \mathbb{Z}_3.$
\end{proposition}

\begin{proof}
The proof is a standard covering space argument. See the proof of Prop.6.4 in \cite[P.39-46]{Ste77}.
\end{proof}

\begin{proposition} \label{Prop: rank of }
Let $\tilde{C}^3$ be the $3$-fold cyclic cover of $C$. Then $\operatorname{Rank}\pi_1(\tilde{C}^3/\partial \tilde{C}^3) \geq 1.$ 
\end{proposition}

\begin{proof}
Standard cyclic cover argument \cite[Ch.6]{Rol76} assures the first homology group $H_1(\tilde{C}^3) \cong \mathbb{Z}_2 \oplus \mathbb{Z}_2 \oplus \mathbb{Z}$. "Modulo out" the generators corresponding the boundary $\tilde{C}^3$ can at most reduce the rank by 2, hence, $\operatorname{Rank}\pi_1(\tilde{C}^3/\partial \tilde{C}^3) \geq 3 - 2 =1$.
\end{proof}

\section{Proof of Theorem \ref{Thm: W^3 embeds in no compact ANR}} \label{section: proof of proposition}
Recall in Section \ref{section: The constructiion of a 3-dimensional example} we pointed out the key in proving Theorem \ref{Thm: W^3 embeds in no compact ANR} is to show that $\operatorname{Rank}\pi_1(K_j, p_1)$ is not bounded. Since $\mathbb{A}_5$ has order 60 and $\Phi_j: \pi_1(K_j,p_1) \to \mathbb{A}_5$ is onto, $\ker \Phi_j$ has index 60 in $\pi_1(K_j,p_1)$. Then the following formula guarantees that it suffices to show that
$\operatorname{Rank}\ker \Phi_j$ is not bounded.

The formula can be viewed as a corollary of the Schreier index theorem. A detailed proof by utilizing covering space theory can be found in \cite[Lemma 1.4]{Ste77}.

\begin{lemma}\label{Lemma: rank bound}
Let $G$ be a group and $H$ be a subgroup of index $i$. If $\operatorname{Rank}H \geq m$, then $\operatorname{Rank}G \geq \frac{m-1}{i}+1$.
\end{lemma}

Let $P_j: (\tilde{K}_j,\tilde{p}_1) \to (K_j,p_1)$ be the covering map such that the induced map $P_{j\ast} : \pi_1(\tilde{K}_j,\tilde{p}_1) \to \pi_1(K_j,p_1)$ is an isomorphism onto $\ker \Phi_j$. By Lemma \ref{Lemma: rank bound}, it remains to show that $\operatorname{Rank}\ker \Phi_j$ is not bounded above as $j\to \infty$, which is equivalent to showing that $\operatorname{Rank}\pi_1(\tilde{K}_j,\tilde{p}_1)\geq 25j $ (resp. $5(5j+1)$) when $j$ is even (resp. odd). The key is the fact that $K_j$ contains $j$ pairwise disjoint incompressible cubes with trefoil-knotted hole. Figure \ref{3_1knot} shows that each $L_l$, $l\geq 1$ contains a cube with trefoil-knotted hole $C_l$. Recall
$$K_j \approx (S^3 \backslash \operatorname{Int}T_j) \cup_{\operatorname{Id}} L_j \cup_{h_{j}^{j-1}} L_{j-1} \cup_{h_{j-1}^{j-2}}\cdots \cup_{h_{2}^{1}} L_1,$$
$K_j$ contains $C_1,C_2,\dots,C_j$, pairwise disjoint cubes with trefoil-knotted hole. The disjointness follows from that each $C_l$ lies in its own $L_l$ and touches only the "inner" boundary of its $L_l$. In $K_j$, when we sew two adjacent $L_l$'s together, only the "outer" boundary of one is glued to the "inner" boundary of the next.

Next, we shall show that $C_l$ in $K_j$ has preimage under the restriction of the covering map $P_j$ has 30 disjoint double covers and 20 disjoint triple covers. The proof heavily relies on the argument given in \cite[P.50-55]{Ste77}. For the convenience of readers, we spell out the proof in details.

Consider $p_l \in C_l$. See Figures \ref{double of 3_1_link} and \ref{cubehole_3_1}. From the Wirtinger presentation (\ref{group_presentation2}), a loop class with subscript $l$ is the class of a loop formed by conjugation of a loop in $L_l$ based at $p_l$ by the path $\mu_l^1$ running from $p_1$ to $p_l$ in $K_j$. Define a change-basepoint isomorphism $M_l: \pi_1(K_j,p_l) \to \pi_1(K_j,p_1)$ generated by conjugation by $\mu_l^1$. By Figures \ref{3_1knot} and \ref{double of 3_1_link}, loop classes $M_l^{-1}(a_l)$, 
$M_l^{-1}(b_l)$ can be viewed as loop classes of $\pi_1(C_l,p_l)$, where $1\leq  l \leq j$. Then 
Figures \ref{double of 3_1_link}-\ref{cubehole_3_1} and Proposition \ref{Prop: presentation of 3_1 knot} assure that the set 
$\{M_l^{-1}(a_l), M_l^{-1}(b_l)\}$ generates $\pi_1(C_l,p_l)$.

Let $\iota_*: \pi_1(C_l,p_l) \to \pi_1(K_j,p_l)$ be the inclusion induced homomorphism. Combine the results from \S \ref{section: The surjection} to obtain the following composition
$$\pi_1(C_l,p_l)\xrightarrow{\iota_*} \pi_1(K_j,p_l)\xrightarrow{M_l} \pi_1(K_j,p_1) \xrightarrow{\Phi_j} \mathbb{A}_5,$$
which has image isomorphic to $\mathbb{Z}_2$ (resp. $\mathbb{Z}_3$) in $\mathbb{A}_5$ when $l = j, j - 2 -4T$ and $j - 4 - 4T$ (resp. $l = j -1-4T$ and $j - 3 - 4T$). See Tables \ref{Table 1}, \ref{Table 3} and \ref{Table 5} (resp. \ref{Table 2} and \ref{Table 4}).  That is because $\Phi_j$ maps $a_l$ and $b_l$ of $\pi_1(C_l,p_l)$ to the same element of order 2 (resp. 3) in $\mathbb{A}_5$. It follows that the kernel of $\Phi_j \circ M_l \circ \iota_2$ has index either 2 or 3 in $\pi_1(C_l,p_l)$. Let $q: (\tilde{C}_l^2,\hat{p}_l) \to (C_l,p_l)$ be a 2-fold cover of $(C_l,p_l)$ corresponding to the kernel.

\begin{claim}
Each $\tilde{C}_l^2$ embeds in $\tilde{K}_j$.
\end{claim} 
\begin{proof}
Note that there exists a lift $\tilde{p}_l$ of $p_l$ in $\tilde{K}_j$ so that $P_{j*}(\pi_1(\tilde{K}_j,\tilde{p}_l)) = \ker (\Phi_j \circ M_l)$. The lift is obtained by lifting $\mu_l^1$ to a path $\tilde{\mu}_l^1$ so $\mu_l^1(0) = \tilde{p}_1$ and the point $\tilde{p}_l$ is defined to be $\tilde{\mu}_{l}^{1}(1)$. Since $\iota_*q_*(\pi_1(\tilde{C}_l^2,\hat{p}_l))\subseteq P_{j*}(\pi_1(\tilde{K}_j,\tilde{p}_l))$, we have the following commutative diagram with $\iota$ lifted to $\tilde{\iota}$

\[%
\begin{array}
[c]{ccc}%
(\tilde{C}_l^2,\hat{p}_l) & \xrightarrow{\tilde{\iota}} & (\tilde{K}_j,\tilde{p}_l)\\
\downarrow q &  & \downarrow P_j \\
(C_l,p_l)  & \xrightarrow{\iota} & (K_j,p_l) 
 
\end{array}
\]
We shall apply standard covering space theory to show $\tilde{\iota}$ is an embedding. It suffices to prove that $\tilde{\iota}$ is 1-1. Suppose $x$ and $y$ are two elements of $\tilde{C}_l^2$ such that $\tilde{\iota}(x) = \tilde{\iota}(y)$. The commutativity of the diagram above implies that
$q(x) = q(y)$. Connect $x$ to $y$ by a path $\alpha$ and $x$ to $\hat{p}_l$ by a path $\beta$ with $\beta(0)= \hat{p}_l$ and $\beta(1) = x$.
Lift $q(\beta)$ to $\tilde{\beta}$ so that $\tilde{\beta}(1) = y$. Suppose $x \neq y$. Then $\tilde{\beta}$ and $\beta$ are distinct lifts of $q(\beta)$. That means $\beta(0) \neq \tilde{\beta}(0)$. So, $\beta \alpha \tilde{\beta}^{-1}$ is not a loop. However, $\tilde{\iota}(\beta \alpha \tilde{\beta}^{-1})$ is a loop in $\tilde{K}_j$. Since $\tilde{\iota}(x) = \tilde{\iota}(y)$, $\tilde{\iota}\beta$ and $\tilde{\iota}\tilde{\beta}$ have to be the same lift of $\iota q(\beta)$. By commutativity of the diagram, $\iota q(\beta \alpha \tilde{\beta}^{-1}) = P_j \tilde{\iota}(\beta \alpha \tilde{\beta}^{-1})$. Hence, $q(\beta \alpha \tilde{\beta}^{-1})$ is a loop in $\iota_*^{-1}P_{j*}(\pi_1(\tilde{K}_j,\tilde{p}_l))$. Thus, $q(\beta \alpha \tilde{\beta}^{-1})$ must lift to a loop at $\hat{p}_l$. Contradiction!
\end{proof}

\begin{remark}
The above argument also works for the 3-fold cover $\tilde{C}_l^3$ which will soon be defined.
\end{remark}

Since $\tilde{\iota}$ is an embedding, $l = j, j - 2 -4T$ and $j - 4 - 4T$, the restriction map $P_j|: \tilde{\iota}(\tilde{C}_l^2) \to C_l$ is a 2-fold cover of $C_l$. Since $\ker \Phi_j$ has index 60 in $\pi_1(K_j)$, the covering space $P_j: \tilde{K}_j \to K_j$ has 60 covering translations. The components of $P_j^{-1}(C_l)$ are the homeomorphic images of $\tilde{\iota}(\tilde{C}_l^2)$ under the 60 covering translations of $P_j$. Thus, every component of $P_j^{-1}(C_l)$ is a 2-fold cover of $C_l$ (i.e., a 2-fold cover of trefoil knot). By \S \ref{section: The constructiion of a 3-dimensional example}, each $K_j$ contains $j$ pairwise disjoint cubes with trefoil-knotted hole $C_l$, where $1\leq l\leq j$. Hence, $\tilde{K}_j$ must have $15j$ (resp. $15(j+1)$) when $j$ is even (resp. odd) pairwise disjoint 2-fold covers of trefoil knot. 

Likewise, let $q': (\tilde{C}_l^3,\hat{p}_l) \to (C_l,p_l)$ be a 3-fold cover of $(C_l,p_l)$ corresponding to the kernel of $\Phi_j \circ M_l \circ \iota_2$. When $l = j -1 -4T$ and $j-3-4T$, the restriction map $P_j|: \tilde{\iota}(\tilde{C}_l^3) \to C_l$ is a 3-fold cover of $C_l$.

\begin{claim}
$P_j|: \tilde{\iota}(\tilde{C}_l^3) \to C_l$ yields a unique $3$-fold (cyclic) cover of $C_l$.
\end{claim}
\begin{proof}
Since the 60-fold covering space of $K_j$ is clearly regular, the restriction of the covering projection to each $C_l$ is also a regular covering. Thus, the induced map $P_{j\ast}|: \tilde{\iota_\ast}(\pi_1(\tilde{C}_l^3)) \to \pi_1(C_l)$ goes onto an index 3 normal subgroup ($\mathbb{Z}_3$). Note that
$\pi_1(\tilde{C}_l^3)$ corresponds to the kernel of the composite $\pi_1(C_l)\xrightarrow{abelianization} \mathbb{Z} \xrightarrow{projection} \mathbb{Z}_3$. Then the claim follows immediately from the uniqueness of the abelianization and the projection.
\end{proof}

When $j$ is even (resp. odd), let $D$ be the complement of the interior of the $15j$ (resp. $15(j+1)$) double covers and $10j$ (resp. $10(j-1)$) triple cover of trefoil knot in $\tilde{K}_j$. Let $Q_j: \tilde{K}_j \to \tilde{K}_j/D$ be quotient map. The quotient space $\tilde{K}_j/D$ is $25j$ (resp. $5(5j+1)$) when $j$ is even (resp. odd) pairwise disjoint 2-fold and 3-fold covers of trefoil knot modulo their boundaries, wedged at the point to which their boundaries are identified. By Propositions \ref{Prop: Rank of 2-fold cover} and \ref{Prop: rank of }, $\pi_1(\tilde{K}_j/D)$ has rank at least $25j$ (resp. $5(5j+1)$) when $j$ is even (resp. odd). Then Propositions \ref{Prop: 2-fold cover} and \ref{Prop: k-fold cover} assure that $Q_j$ induces a surjection of $\pi_1(\tilde{K}_j)$ onto $\pi_1(\tilde{K}_j/D)$, hence, $\operatorname{Rank}\pi_1(\tilde{K}_j) \geq 25j$ (resp. $5(5j+1)$) when $j$ is even (resp. odd).

This completes the proof of Theorem \ref{Thm: W^3 embeds in no compact ANR}.

\begin{proof}[Proof of Theorem \ref{Thm: high dimensional collection}]
Using our building block $W^3$, one can apply the standard "drilling tunnel" and "piping" to generate high-dimensional examples $W^n$. We only spell out an outline. A detailed proof described in \cite[P.56-62]{Ste77} can readily be applied.

Recall in \S \ref{section: A presentation} there is an arc $\mu_l^1$ connecting the base points $p_l \in \partial T_l'$ and $q_l \in \partial T_l$ (see Figure \ref{double of 3_1_link}). The sewing homeomorphism $h_{l+1}^{l}$ identifies $q_l$ with $p_{l+1}$. By the construction of $W^3$, those arcs fit together to form a (base) ray $R$ in $W^3$. Find a regular neighborhood $N$ of $R$ such that $W^+ = W^3 \backslash \operatorname{Int}N$ is a PL manifold with
$\partial W^+$ homeomorphic to $\mathbb{R}^2$ and $\operatorname{Int}W^+$ homeomorphic to $W^3$. The $n$-dimensional example
$W^n$ is defined to be $W^n = \partial (B^{n-2}\times W^+)= (B^{n-2}\times \partial W^+) \cup (\partial B^{n-2}\times W^+)$, where $B^{n-2}$ is a codimension 2 ball. The openness and contractibility follow from the standard PL topology arguments. 

Define solid torus $T_l^+$ a subset of $W^+$ by $T_l^+ = T_l^* \backslash \operatorname{Int}N$. Then $W^+$ can be expressed by $\cup T_l^+$. Let $p_2: B^{n-2} \times W^+ \to W^+$ be a projection sending $B^{n-2}\times W^+$ onto its second factor. Let $p: W^n \to W^+$ be the restriction of $p_2$. Suppose there is a compact, locally connected, locally 1-connected metric space $U$ that contains $W^n$ as an open set. Then it suffices to show $\pi_1(U\backslash p^{-1}(\operatorname{Int}T_0^+))$ is not finitely generated just as how we prove Theorem \ref{Thm: W^3 embeds in no compact ANR}. By definition of $N$, $T_0^+ = T_0^*$. Let $q$ be the quotient map 
$$q: T_j^+ \backslash \operatorname{Int}T_0^+ \to (T_j^+ \backslash \operatorname{Int}T_0^+)/\partial T_j^+.$$
Extend $q$ to map $Q: U \backslash p^{-1}(\operatorname{Int}T_0^+) \to (T_j^+ \backslash \operatorname{Int}T_0^+)/\partial T_j^+$. There should be no difficult in doing so because $U \backslash p^{-1}(\operatorname{Int}T_0^+)$ can be decomposed into the union of 
$U \backslash p^{-1}(\operatorname{Int}T_j^+)$ and $p^{-1}(T_j^+ \backslash \operatorname{Int}T_0^+)$. Then $Q$ can be defined as the union of the constant map $l: U \backslash p^{-1}(\operatorname{Int}T_j^+) \to (T_j^+ \backslash \operatorname{Int}T_0^+)/\partial T_j^+$ and
the restriction map $q \circ p|_{p^{-1}(T_j^+ \backslash \operatorname{Int}T_0^+)}$. By Lemma \ref{lemma: collar}, $q \circ p|_{p^{-1}(T_j^+ \backslash \operatorname{Int}T_0^+)}$ induces a surjection on fundamental groups, so does $Q$.  Note that $(T_j^+ \backslash \operatorname{Int}T_0^+)/\partial T_j^+$ and 
$(T_j^* \backslash \operatorname{Int}T_0^*)/\partial T_j^*$ are homeomorphic. Thus,  showing that $\operatorname{Rank}\pi_1(U\backslash p^{-1}(\operatorname{Int}T_0^+))$ has no lower bound is equivalent to proving
$\operatorname{Rank}\pi_1\big((T_j^*\backslash \operatorname{Int}T_0^*)/\partial T_j^*\big) = \operatorname{Rank}\pi_1(K_j)$, which is just an application of Theorem \ref{Thm: W^3 embeds in no compact ANR}.
\end{proof}

\section{Questions}\label{section: questions}
Recall the construction of $W^3$ in \S \ref{section: The constructiion of a 3-dimensional example}
\begin{equation}\label{Decomposition of W^3}
W^3 = \lim_{j\to \infty} L_j \cup_{h_{j}^{j-1}} L_{j-1} \cup_{h_{j-1}^{j-2}}\cdots \cup_{h_{2}^{1}} L_1,
\end{equation}
where the sewing homeomorphism $h_{l+1}^{l}$ identifies the boundary component $\partial T_l$ of $L_l$ to the boundary component $\partial T'_{l+1}$
of $L_{l+1}$. Unknotting the cube with trefoil-knotted hole as shown in Figure \ref{3_1knot} results in a cobordism $L^\ast$, which is widely known as the first stage of constructing a Whitehead manifold. See Figure \ref{whitehead}.

\begin{figure}[h!]
        \centering
       \includegraphics[ width=7cm, height=9cm]{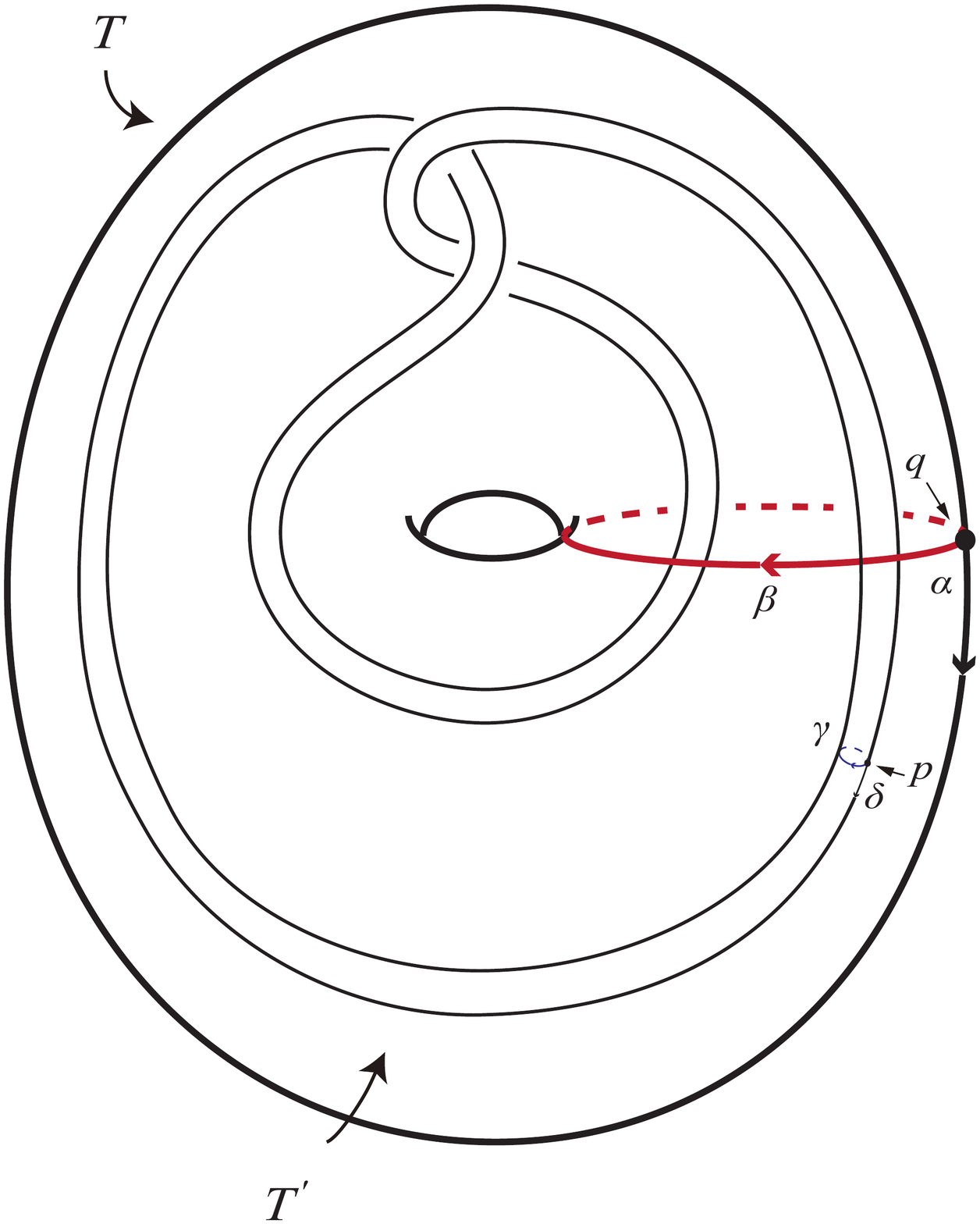}
       \caption{$L^\ast = T \backslash T'$. The "inner" boundary component of $L^\ast$ is $\partial T'$. The "outer" boundary component 
       of $L^\ast$ is $\partial T$.}
        \label{whitehead}
\end{figure}

Consider a variation of $W^3$ by placing $L^*$ ahead of $L_j$ or inserting $L^*$ between adjacent $L_l$ and $L_{l+1}$ in (\ref{Decomposition of W^3})
\begin{equation}\label{Decomposition2}
W^{\ast} = \lim_{j\to \infty} L_j \cup_{H_{j}^{\ast}} L^\ast \cup_{H_{\ast}^{j-1}}L_{j-1}\cdots \cup_{h_{2}^{1}} L_1,
\end{equation}
where the sewing homeomorphism $H_{*}^{l}$ identifies the boundary component $\partial T_l$ of $L_l$ to the boundary component $\partial T'$
of $L^\ast$ and the sewing homemorphism  $H_{l+1}^{*}$ identifies the boundary component $\partial T$ of $L^\ast$ to the boundary component $\partial T'_{l+1}$ of $L_{l+1}$. Then we obtain an infinite collection $\mathcal{C}$ by inserting $L^*$'s in (\ref{Decomposition of W^3}).

The following result is an example of $\mathcal{C}$.
\begin{proposition}
The $3$-dimensional example $W$ constructed by Sternfeld belongs to the collection $\mathcal{C}$.
\end{proposition}

\begin{proof}
The manifold $W$ constructed by Sterneld is homeomorphic to $L^\ast \cup_{H_{*}^{j}} L_{j} \cup_{H_{j}^{*}}L^\ast \cdots $, i.e., inserting $L^*$ in (\ref{Decomposition of W^3}) every other slot. See Figure \ref{sternfeld}. If one ignores the grey curves as shown in Figure \ref{sternfeld}, then the picture will be exactly the same picture given in \cite[P.4]{Ste77}. In other words, solid tori $T$ and $T_{j-1}'$ are the first stage of Sternfeld's construction. 
\end{proof}

\begin{remark}
Let $K_j$ and $K_i$ be the corresponding knot spaces of $W^3$ and $W$ respectively. Although both $W^3$ and $W$ contain a cube with a trefoil-knotted hole at each stage of the construction, the corresponding 60-fold covers of $K_j$ and $K_i$ are different. That is, the 60-fold cover of $K_j$ has both embedded 2-fold covers and embedded 3-fold covers of incompressible cube with a trefoil-knotted hole in $K_j$. However, the 60-fold cover of $K_i$ has only embedded 2-fold covers of incompressible cube with trefoil-knotted hole in $K_i$.
\end{remark}

\begin{figure}[h!]
        \centering
       \includegraphics[ width=6.8cm, height=8cm]{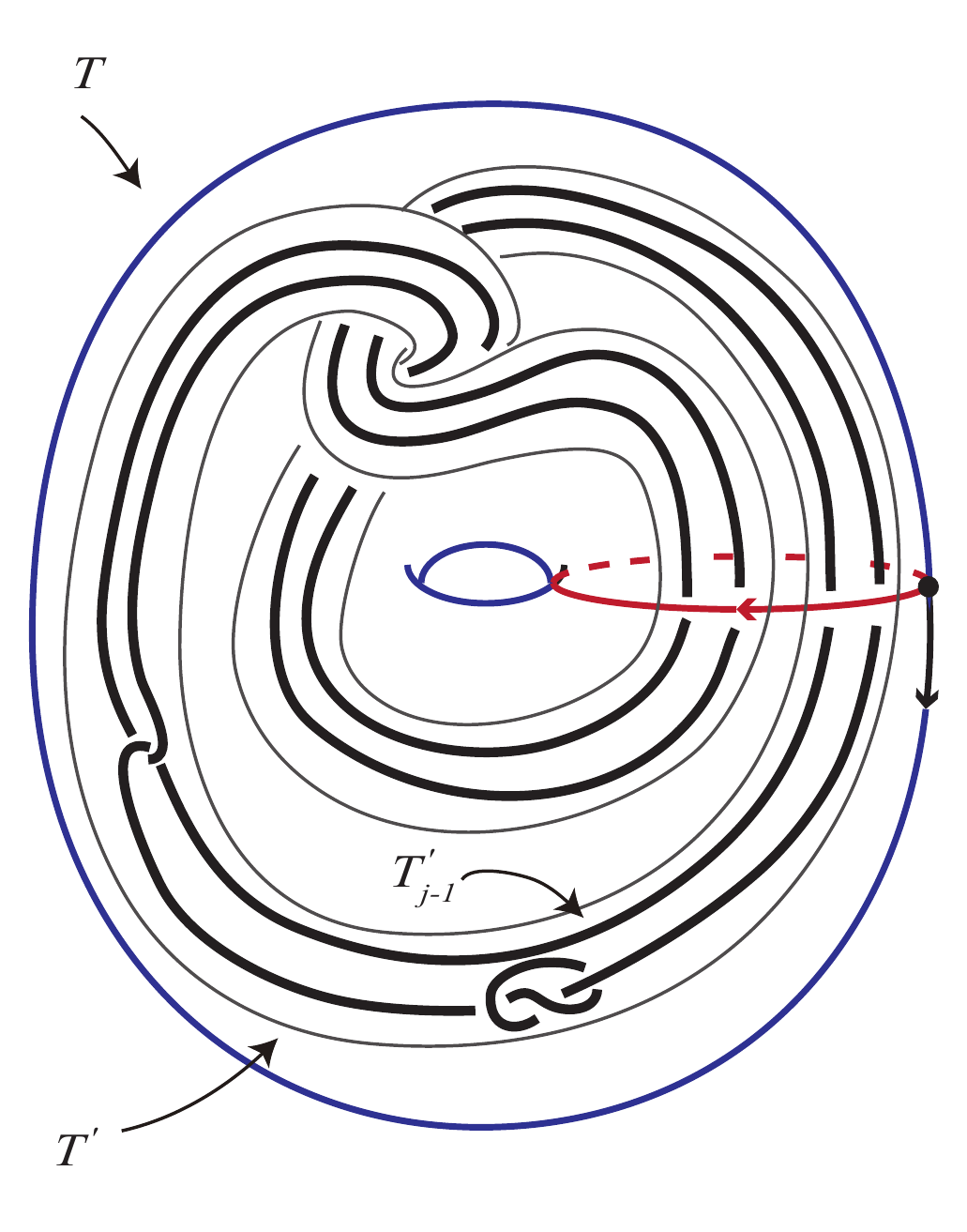}
       \caption{The difference between solid torus $T$ (blue) and $T'$ (grey) is $L^\ast$. This $L_{j-1}$ is the area between
       $\partial T_{j-1}$ (which has been identified with $\partial T'$) and $\partial T_{j-1}'$. }
        \label{sternfeld}
\end{figure}

\begin{question}
Does $\mathcal{C}$ contain an infinite subcollection of contractible open 3-manifolds $\mathcal{C}'$ such that each manifold in $\mathcal{C}'$ embeds in no compact, locally connected and locally 1-connected metric $3$-space?
\end{question}

\begin{question}
The cube with trefoil-knotted hole $C_l$ plays the key role in this paper. Let $K$ be an arbitrary (nontrivial) knot. Can $C_l$ be replaced by a cube with a $K$-knotted hole? More specifically, if we replace $C_l$ at each stage in the construction of $W^3$ by cube with a $K$-knotted hole, can the resulting contractible open manifold $W'$ embed in some compact, locally connected and locally 1-connected metric 3-space?
\end{question}

\section*{Acknowledgements}
I would like to thank Professor Craig Guilbault for bringing Bing's and Sternfeld's examples to my attention and many helpful discussions on this work. I also thank the referee for the comments and for giving this paper a very close reading.

\end{document}